\documentclass[12pt, final]{article}
     
\RequirePackage{amsfonts, amsmath} % AMS math environments (fold)

\RequirePackage{graphicx}                  % \includegraphics  -  jpg, png (partly pdf) (fold)
\RequirePackage{algorithm, algorithmic}    % provides algorithm environment (fold)
                                           
\RequirePackage{color, colortbl}           % color and colored table-background (fold)
\RequirePackage{enumitem}

\usepackage{amsopn}
\RequirePackage{amsthm}
\RequirePackage{amsbsy}
\theoremstyle{plain} % {plain} {definition} {remark}
\newtheorem{theorem}{Theorem}[section]

\newtheorem{lemma}      [theorem]{Lemma}

\newtheorem*{corollary}{Corollary}

\newcommand{\e}   {\textnormal{e}}                         % Eulerian number
\DeclareMathOperator*{\argmin}{arg\,min}                   % arg min
\renewcommand{\t} {^{\top}}                                % transpose, e.g. $A\t$
\renewcommand{\d} {{\rm d}}                                % differential d
\newcommand{\F} {{\rm F}}                                  % Frobenius
\newcommand{\norm} [2][]{\left\|#2\right\|_{#1}}           % f.e. \norm[\infty]{A-B}
\newcommand{\rank} [1]  {{\rm rank}\left( #1 \right)}    % rank of a matrix
\newcommand{\trace} [1]  {{\rm tr\!}\left( #1 \right)}    % rank of a matrix
\newcommand{\diag} [1]  {{\rm diag\!}\left( #1 \right)}    % diagonal of a matrix
\newcommand{\fro} {\rm F}                                  % define Frobenius norm
\newcommand{\cov} [1]{{\rm Cov}\!\left[ #1 \right]}         % f.e. \cov{\bfA}
\newcommand{\bfGamma}{{\boldsymbol{\Gamma}}}

\newcommand{\bfLambda}{{\boldsymbol{\Lambda}}}

\newcommand{\bfSigma}{{\boldsymbol{\Sigma}}}

\newcommand{\bfPsi}{{\boldsymbol{\Psi}}}

\newcommand{\bfdelta}{{\boldsymbol{\delta}}}
\newcommand{\bfepsilon}{{\boldsymbol{\epsilon}}}

\newcommand{\bfmu}{{\boldsymbol{\mu}}}

\newcommand{\bfxi}{{\boldsymbol{\xi}}}

\newcommand{\bfA}{{\bf A}}
\newcommand{\bfB}{{\bf B}}
\newcommand{\bfC}{{\bf C}}
\newcommand{\bfD}{{\bf D}}

\newcommand{\bfF}{{\bf F}}
\newcommand{\bfG}{{\bf G}}
\newcommand{\bfH}{{\bf H}}
\newcommand{\bfI}{{\bf I}}

\newcommand{\bfK}{{\bf K}}
\newcommand{\bfL}{{\bf L}}
\newcommand{\bfM}{{\bf M}}

\newcommand{\bfP}{{\bf P}}
\newcommand{\bfQ}{{\bf Q}}
\newcommand{\bfR}{{\bf R}}
\newcommand{\bfS}{{\bf S}}

\newcommand{\bfU}{{\bf U}}
\newcommand{\bfV}{{\bf V}}
\newcommand{\bfW}{{\bf W}}
\newcommand{\bfX}{{\bf X}}
\newcommand{\bfY}{{\bf Y}}
\newcommand{\bfZ}{{\bf Z}}

\newcommand{\bfb}{{\bf b}}

\newcommand{\bfe}{{\bf e}}

\newcommand{\bfq}{{\bf q}}
\newcommand{\bfr}{{\bf r}}

\newcommand{\bfu}{{\bf u}}
\newcommand{\bfv}{{\bf v}}
\newcommand{\bfw}{{\bf w}}
\newcommand{\bfx}{{\bf x}}
\newcommand{\bfy}{{\bf y}}

\newcommand{\bfzero}{{\bf0}}
\newcommand{\bfone}{{\bf1}}

\newcommand{\calN}{\mathcal{N}}

\newcommand{\bbE}{\mathbb{E}}

\newcommand{\bbR}{\mathbb{R}}

\newcommand{\rel}{{\rm rel}}                     % vectorize matrix

\usepackage{fullpage}
\usepackage{marvosym}
\newcommand{\TheTitle}{Optimal regularized inverse matrices\\  for \\ inverse problems}

\title{{\TheTitle}}

\author{
  Julianne Chung
	\thanks{Department of Mathematics, Virginia Tech, Blacksburg, VA \newline \hspace*{10ex}
    \Letter \ \texttt{\href{mailto:jmchung@vt.edu}{jmchung@vt.edu} \ \ \Mundus \ \href{http://www.math.vt.edu/people/jmchung/}{www.math.vt.edu/people/jmchung/}}}
  \and
  Matthias Chung
	\thanks{Department of Mathematics, Virginia Tech, Blacksburg, VA \newline \hspace*{10ex}
    \Letter \ \texttt{\href{mailto:mcchung@vt.edu}{mcchung@vt.edu} \ \ \Mundus \ \href{http://www.math.vt.edu/people/mcchung/}{www.math.vt.edu/people/mcchung/}}}
}

\begin{document}

\maketitle

% REQUIRED
\begin{abstract}
In this paper, we consider optimal low-rank regularized inverse matrix approximations and their applications to inverse problems.  We give an explicit solution to a generalized rank-constrained regularized inverse approximation problem, where the key novelties are that we allow for updates to existing approximations and we can incorporate additional probability distribution information. Since computing optimal regularized inverse matrices under rank constraints can be challenging, especially for problems where matrices are large and sparse or are only accessable via function call, we propose an efficient rank-update approach that decomposes the problem into a sequence of smaller rank problems.  Using examples from image deblurring, we demonstrate that more accurate solutions to inverse problems can be achieved by using rank-updates to existing regularized inverse approximations. Furthermore, we show the potential benefits of using optimal regularized inverse matrix updates for solving perturbed tomographic reconstruction problems.
\end{abstract}

% REQUIRED

% REQUIRED
\noindent{\it AMS}: 
 	65F22, 15A09, 15A29

\section{Introduction} \label{sec:introduction} % (fold)
Optimal low-rank inverse approximations play a critical role in many scientific applications such as matrix completion, machine learning, and data analysis \cite{Ye2005,Drineas2007,Markovsky2012}.  Recent theoretical and computational developments on \emph{regularized} low-rank inverse matrices have enabled new applications, such as for solving inverse problems \cite{Chung2015}. In this paper, we develop theoretical results for a general case for finding \emph{optimal regularized inverse matrices} (ORIMs), and we propose novel uses of these matrices for solving 
linear ill-posed inverse problems of the form,
\begin{equation}
	\label{eqn:linearsystem}
	\bfb = \bfA \bfxi +\bfdelta,
\end{equation}
where $\bfxi \in \bbR^n$ is the desired solution, $\bfA \in \bbR^{m \times n}$ models the forward process, $\bfdelta \in \bbR^m$ is additive noise, and $\bfb \in \bbR^m$ is the observed data.  We assume that $\bfA$ is very large and sparse, or that $\bfA$ cannot be formed explicitly, but matrix vector multiplications with $\bfA$ are feasible (e.g., $\bfA$ can be an object or function handle).  Furthermore, we are interested in ill-posed inverse problems, whereby small errors in the data may result in large errors in the solution \cite{Hadamard1923,Hansen2010,Vogel1987}, and regularization is needed to stabilize the solution.

Next, we provide a brief introduction to regularization and ORIMs, followed by a summary of the main contributions of this work.
 Various forms of regularization have been proposed in the literature, including variational methods \cite{RuOsFa92,tikhonov1977solutions} and iterative regularization, where early termination of an iterative methods provides a regularized solution \cite{HaHa93,Hank95a}.
Optimal regularized inverse matrices have been proposed for solving inverse problems and have been studied in both the Bayes and empirical Bayes framework \cite{Chung2011,Chung2013a,Chung2015}. 
Let $\bfP \in \bbR^{n \times m}$ be an initial approximation matrix (e.g., $\bfP=\bfzero_{n \times m}$ in previous works). 
Then treating $\bfxi$ and $\bfdelta$ as random variables, the goal is to find a matrix $\widehat\bfZ \in \bbR^{n \times m}$ that gives a small reconstruction error.  That is, $\rho((\bfP+\widehat\bfZ) \bfb - \bfxi)$ should be small for some given error measure $\rho:\bbR^n \to \bbR^+_0$.  
% Particular choices of $\rho$ and distributions of $\bfxi$ and $\bfdelta$ determine the regularization matrix $\widehat\bfZ$.
In this paper, we consider $\rho$ to be the squared Euclidean norm, and we seek an \emph{optimal} matrix $\widehat\bfZ$ that minimizes the expected value of the errors with respect to the joint distribution of $\bfxi$ and $\bfdelta$.  Hence, the problem of finding an ORIM $\widehat\bfZ$ can be formulated as 
\begin{equation}  
	\label{eqn:Bayesmin}
\widehat\bfZ = \argmin_{\bfZ} \ \bbE\, \norm[2]{((\bfP+ \bfZ) \bfA - \bfI_n) \bfxi + \bfZ \bfdelta}^2\,.
\end{equation} 
This problem is often referred to as a \emph{Bayes risk minimization problem} \cite{Carlin2000,Vapnik1998}. 
%but as we will see in Section~\ref{sec:background}, it is equivalent to a matrix approximation problem.
Especially for large scale problems, it may be advisable to include further constraints on $\bfZ$ such as sparsity, symmetry, block or cyclic structure, or low-rank structure. Here, we will focus on matrices $\bfZ$ of low-rank. Once computed, ORIM $\widehat\bfZ$ has mainly been used to efficiently solve linear inverse problems in an online phase as data $\bfb$ becomes available and requires therefore only a matrix-vector multiplication $(\bfP +\widehat\bfZ) \bfb$.  

% subsection regularization_for_linear_inverse_problems (end)

% \subsection{Overview of our contributions} % (fold)
% \label{sub:overview_of_our_contributions}
\paragraph{Overview of our contributions}
First, we derive a closed-form solution for problem~\eqref{eqn:Bayesmin} under rank constraints with uniqueness conditions.  The two key novelties are that we include matrix $\bfP$, thereby allowing for updates to existing regularized inverse matrices, and we incorporate additional information regarding the distribution of $\bfxi$.  More specifically, we allow non-zero mean for the distribution of $\bfxi$ and show that our results reduce to previous results in \cite{Chung2015} that assume zero mean and $\bfP=\bfzero_{n \times m}$.  These extension are not trivial and require a different approach than \cite{Chung2015} for the proof.  Second, we describe an efficient rank-update approach for computing a global minimizer of~\eqref{eqn:Bayesmin} under rank constraints, that is related to but different than the approach described in \cite{chung2014efficient} where training data was used as a substitute for knowledge of the forward model. We demonstrate the efficiency and accuracy of the rank-update approach, compared to standard SVD-based methods, for solving a sequence of ill-posed problems.

Third, we propose novel uses of ORIM updates in the context of solving inverse problems.  An example from image deblurring demonstrates that updates to existing regularized inverse matrix approximations such as the Tikhonov reconstruction matrix can lead to more accurate solutions.  
Also, we use an example from tomographic image reconstruction to show that ORIM updates can be used to efficiently and accurately solve perturbed inverse problems.  This contribution has significant implications for further research development, ranging from use within nonlinear optimization schemes to preconditioner updates.

The key benefits of using ORIMs for solution updates and for solving inverse problems are that (1) we approximate the regularized inverse directly, so reconstruction or application requires only a matrix-vector multiplication rather than a linear solve; (2) our matrix inherently incorporates regularization; (3) ORIMs and ORIM updates can be computed for any general rectangular matrix $\bfA$, even if $\bfA$ is only available via a function call, making it ideal for large-scale problems. 

The paper is organized as follows.  In Section~\ref{sec:background}, we provide preliminaries to establish notation and summarize important results from the literature. Then, in Section~\ref{sec:proof_general_case}, we derive a closed form solution to problem~\eqref{eqn:Bayesmin} under rank constraints and provide uniqueness conditions (see Theorem~\ref{thm:mainresult} for the main result).
For large-scale problems, computing an ORIM according to Theorem~\ref{thm:mainresult} may be computationally prohibitive, so in Section~\ref{sub:computational_methods_for_obtaining_bfz}, we describe a rank-update approach for efficient computation.
Finally, in Section~\ref{sec:numerics} we provide numerical examples from image processing that demonstrate the benefits of ORIM updates.
 Conclusions and discussions are provided in Section~\ref{sec:conclusions}.

% subsection overview_of_our_contributions (end)

\section{Background} % (fold)
\label{sec:background}
In this section, we begin with preliminaries to establish notation.

Given a matrix $\bfA\in\bbR^{m\times n}$ with rank $k \leq \min(m,n)$, let $\bfA= \bfU_\bfA \bfSigma_\bfA \bfV_\bfA\t$ denote the singular value decomposition (SVD) of $\bfA$, where $\bfU_\bfA = [\bfu_1,\ldots,\bfu_m] \in \bbR^{m \times m}$ and $\bfV_\bfA = [\bfv_1,\ldots,\bfv_n] \in \bbR^{n \times n}$ are orthogonal matrices that contain the left and right singular vectors of $\bfA$, respectively. Diagonal matrix $\bfSigma_\bfA = \diag{\sigma_1(\bfA),\ldots,\sigma_k(\bfA), 0,\ldots,0}\in \bbR^{m \times n}$ contains the singular values $\sigma_1(\bfA)\geq \cdots \geq \sigma_k(\bfA) > 0$ and zeros on its main diagonal. The truncated SVD approximation of rank $r\leq k$ of $\bfA$ is denoted by $\bfA_r = \bfU_{\bfA,r} \bfSigma_{\bfA, r} \bfV_{\bfA,r}\t \in \bbR^{m\times n}$ where $\bfU_{\bfA,r}$ and $\bfV_{\bfA,r}$ contain the first $r$ vectors of $\bfU_{\bfA}$ and $\bfV_{\bfA}$ respectively, and $\bfSigma_{\bfA,r}$ is the principal $r \times r$ submatrix of $\bfSigma_\bfA$.  The TSVD approximation is unique if and only if $\sigma_r(\bfA)>\sigma_{r+1}(\bfA)$.  Furthermore, the Moore-Penrose pseudoinverse of $\bfA$ is given by $\bfA^\dagger = \bfV_{\bfA,k} \bfSigma_{\bfA,k}^{-1} \bfU_{\bfA,k}\t$.
% , and we denote $\bfR_\bfA = \bfV_{\bfA,k} \bfV_{\bfA,k}\t \in \bbR^{n \times n}$ to be the sum of outer products.

Next we show that the problem of finding an \emph{optimal regularized inverse matrix} (ORIM) (i.e., a solution to~\eqref{eqn:Bayesmin}) is equivalent to solving a matrix approximation problem.  That is, assuming $\bfxi$ and $\bfdelta$ are random variables, the goal is to find a matrix $\bfZ$ such that we minimize the expected value of the squared 2-norm error, i.e., $\min_\bfZ f(\bfZ)$, where
\begin{equation*}
	f(\bfZ) =  \bbE\, \norm[2]{(\bfP+\bfZ)\bfb -\bfxi}^2 = \bbE\, \norm[2]{(\bfP+\bfZ)(\bfA\bfxi+\bfdelta) -\bfxi}^2
\end{equation*}
is often referred to as the \emph{Bayes risk}. 

Lets further assume that $\bfxi$ and $\bfdelta$ are independent random variables with $\bbE[\bfxi] = \bfmu_\bfxi$, the covariance matrix $\cov{\bfxi} = \bfGamma_\bfxi$ is symmetric positive definite, $\bbE[\bfdelta] = \bfzero_{m\times 1}$, and $\cov{\bfdelta} = \eta^2\bfI_m$.
First, due to the independence of $\bfxi$ and $\bfdelta$ and since $\bbE[\bfdelta] = \bfzero_{m\times 1}$, we can rewrite the Bayes risk as
$$ f(\bfZ) = \bbE\, \left[\norm[2]{((\bfP+\bfZ)\bfA-\bfI_n)\bfxi}^2\right] + \bbE\, \left[\norm[2]{(\bfP+\bfZ)\bfdelta}^2\right].$$
Then using the property of the quadratic form \cite{Seber2012}, $\bbE \left[\bfepsilon\t \bfLambda \bfepsilon\right] = \trace{\bfLambda\bfSigma_{\bfepsilon}} + \bfmu_{\bfepsilon}\t \bfLambda \bfmu_{\bfepsilon}$, 
where $\rm tr(\cdot)$ denotes the trace, $\bfLambda$ is symmetric, $\bbE[\bfepsilon] = \bfmu_\bfepsilon$ and $\cov{\bfepsilon} = \bfSigma_\bfepsilon$,
\begin{align*}
	f(\bfZ) =& \, \bfmu_\bfxi\t((\bfP+\bfZ)\bfA -\bfI_n)\t((\bfP+\bfZ)\bfA -\bfI_n)\bfmu_\bfxi \\
	&+ \trace{((\bfP+\bfZ)\bfA -\bfI_n)\t((\bfP+\bfZ)\bfA -\bfI_n)\bfM_\bfxi\bfM_\bfxi\t} +  \eta^2\, \trace{(\bfP+\bfZ)\t(\bfP+\bfZ)}
\end{align*}
with $\bfM_\bfxi\bfM_\bfxi\t = \bfGamma_\bfxi$ being any symmetric factorization, e.g., Cholesky factorization. Using the cyclic property of the trace leads to
$$
f(\bfZ) = \norm[2]{((\bfP+\bfZ)\bfA -\bfI_n)\bfmu_\bfxi}^2 + \norm[\fro]{((\bfP+\bfZ)\bfA -\bfI_n)\bfM_\bfxi}^2 + \eta^2\norm[\fro]{(\bfP+\bfZ)}^2, 
$$
where $\norm[\fro]{\,\cdot\,}$ denotes the Frobenius norm.  Next we rewrite $f(\bfZ)$ in terms of only one Frobenius norm. Let $\bfM = \begin{bmatrix} \bfM_\bfxi & \bfmu_\bfxi\end{bmatrix} \in \bbR^{n \times (n+1)}$, then using the identities of the Frobenius and the vector 2-norm, as well as applying Kronecker product properties, we get
\begin{equation}\label{eq:reformulatedFcn}
f(\bfZ) = \norm[\fro]{ \bfZ \begin{bmatrix}
	\bfA \bfM & \eta \bfI_m \end{bmatrix} - \begin{bmatrix}
		\bfM - \bfP \bfA \bfM & -\eta \bfP
	\end{bmatrix}}^2.	
\end{equation}
 
Thus, minimizing the Bayes risk in problem~\eqref{eqn:Bayesmin} is equivalent to minimizing~\eqref{eq:reformulatedFcn}.
Notice that so far we have not imposed any constraints on $\bfZ$.  Although various constraints can be imposed on $\bfZ$, here we consider $\bfZ$ to be of low-rank, i.e., $\rank{\bfZ}\leq r$ for some $r \leq \rank{\bfA}$. Hence the low-rank matrix approximation problem of interest in this paper is
\begin{equation}
	\label{eqn:lrproblem}
	\min_{\rank{\bfZ} \leq r} \,\,f(\bfZ) = \norm[\fro]{ \bfZ \begin{bmatrix}
	\bfA \bfM & \eta \bfI_m \end{bmatrix} - \begin{bmatrix}
		\bfM - \bfP \bfA \bfM & -\eta \bfP
	\end{bmatrix}}^2.
\end{equation}
We will provide a closed form solution for~\eqref{eqn:lrproblem} in Section~\ref{sec:proof_general_case}, but it is important to remark that special cases of this problem have been previously studied in the literature.  For example, a solution for the case where $\bfP=\bfzero_{n \times m}$ and $\bfmu_\bfxi = \bfzero_{n\times 1}$ was provided in \cite{Chung2015} that uses the generalized SVD of $\left\{\bfA, \bfM_\bfxi^{-1}\right\}$.
If, in addition, we assume $\bfM_\bfxi = \bfI_n,$ then an optimal regularized inverse matrix of at most rank $r$ reduces to a truncated-Tikhonov matrix \cite{Chung2015},
\begin{equation}
	\label{eqn:TTik}
	\widehat \bfZ = \bfV_{\bfA,r} \bfPsi_{\bfA,r} \bfU_{\bfA,r}\t,
\end{equation}
where $\bfPsi_{\bfA,r} = \diag{\frac{\sigma_1(\bfA)}{\sigma_1^2(\bfA)+ \eta^2}, \ldots, \frac{\sigma_r(\bfA)}{\sigma_r^2 (\bfA)+ \eta^2}}$. Moreover, this $\widehat \bfZ$ is the {\em unique} global minimizer for 
\begin{equation} \label{eq:objFcn}
\min_{\rank{\bfZ} \leq r} \ % f(\bfZ) =
\norm[\fro]{\bfZ \bfA - \bfI_n}^2 + \eta^2  \norm[\fro]{\bfZ}^2,
\end{equation}
 if and only if $\sigma_r(\bfA) > \sigma_{r+1}(\bfA)$.

% section background (end)

\section{Low-rank optimization problem}  \label{sec:proof_general_case} % (fold)
The goal of this section is to derive the unique global minimizer for problem~\eqref{eqn:lrproblem}, under suitable conditions.  We actually consider a more general problem, as stated in Theorem~\ref{thm:mainresult}, where $\bfM \in \bbR^{n \times p}$ with $\rank{\bfM} = n \leq p.$
Our proof uses a special case of Theorem~2.1 from Friedland \& Torokhti \cite{Friedland2007} that is provided here for completeness.

\begin{theorem}\label{thm:Friedland} Let matrices $\bfB \in \bbR^{m\times n}$ and $\bfC \in \bbR^{q \times n}$ with $k = \rank{\bfC}$ be given. Then
	$$ \widehat \bfZ  = \left(\bfB \bfV_{\bfC,k}\bfV_{\bfC,k}\t \right)_r \bfC^\dagger$$
	is a solution to the minimization problem
	$$ \min_{\rank{\bfZ} \leq r}  \norm[\F]{\bfZ\bfC- \bfB}^2,$$
	having a minimal $\norm[\F]{\bfZ}$. This solution is unique if and only if either
	$$ r \geq \rank{\bfB\bfV_{\bfC,k}\bfV_{\bfC,k}\t} $$
	or
	$$ 1 \leq r < \rank{\bfB\bfV_{\bfC,k}\bfV_{\bfC,k}\t} \quad \mbox{and} \quad \sigma_r(\bfB\bfV_{\bfC,k}\bfV_{\bfC,k}\t) > \sigma_{r+1}(\bfB\bfV_{\bfC,k}\bfV_{\bfC,k}\t).$$
\end{theorem}

% Proof for Friedland and Torokhti
\begin{proof}
	See \cite{Friedland2007}.
	\end{proof}
To get to our main result we first provide the following Lemma.
\begin{lemma}\label{lem:svdExtended}
Let $\bfB = [\bfA  \ \  \eta \, \bfI_m]$ with $\bfA \in \bbR^{m\times n}$ and parameter $\eta \geq 0$, nonzero if $\rank{\bfA} < \max\{m,n\}$. Let further $\bfD_\bfA\in\bbR^{m \times m}$ with 
$\bfD_\bfA = \diag{\sqrt{\sigma_1^2(\bfA)+\eta^2},\ldots, \sqrt{\sigma_n^2(\bfA)+\eta^2}, \eta, \ldots, \eta}$  for  $m\geq n$ and 
$\bfD_\bfA = \diag{\sqrt{\sigma_1^2(\bfA)+\eta^2},\ldots, \sqrt{\sigma_m^2(\bfA)+\eta^2}}$ for $m < n$.
Then the SVD of $\bfB$ is given by $\bfB= \bfU_\bfB \bfSigma_\bfB \bfV_\bfB\t$, where 
$$\bfU_\bfB = \bfU_\bfA, \quad \bfSigma_\bfB = \left[ \bfD_\bfA \ \ \bfzero_{m \times n} \right] \quad \mbox{and} \quad \bfV_\bfB = \begin{bmatrix}
	\bfV_\bfA \bfSigma_\bfA\t\bfD_\bfA^{-1} & \bfV_{12} \\
	\eta \, \bfU_\bfA\bfD_\bfA^{-1}         & \bfV_{22}
\end{bmatrix},$$
with arbitrary $\bfV_{12}$ and $\bfV_{22}$ satisfying $\bfV_{12}\t\bfV_{12}+\bfV_{22}\t\bfV_{22} = \bfI_n$ and $\bfA\bfV_{12}+\eta\bfV_{22} = \bfzero_{m\times n}.$
\end{lemma}

% Proof for lemma
\begin{proof}
Let the SVD of $\bfA =\bfU_\bfA \bfSigma_\bfA \bfV_\bfA\t$ be given. First, notice that the singular values  $\sigma_j(\bfB) = \sqrt{\lambda_j(\bfB\bfB\t)}$, where $\lambda_j(\bfB \bfB\t)$ defines the $j$-th eigenvalue of the matrix $\bfB \bfB\t$ with $\lambda_1(\bfB \bfB\t)\geq \cdots \geq \lambda_n(\bfB \bfB\t)$. Since the eigenvalue decomposition of $\bfB\bfB\t$ is given by
\begin{equation} \label{eq:eigendecompB}
	\bfB\bfB\t = \bfU_\bfA (\bfSigma_\bfA \bfSigma_\bfA\t  +\eta^2 \bfI_m) \bfU_\bfA\t
\end{equation}
we have 
$$
	\bfSigma_\bfB = \left[ \bfD_\bfA \ \ \bfzero_{m \times n} \right]
$$ 
with $\bfD_\bfA$, where 
$$\bfD_\bfA = \diag{\sqrt{\sigma_1^2(\bfA)+\eta^2},\ldots, \sqrt{\sigma_n^2(\bfA)+\eta^2}, \eta, \ldots, \eta} \quad \mbox{if } m\geq n, $$
and
$$\bfD_\bfA = \diag{\sqrt{\sigma_1^2(\bfA)+\eta^2},\ldots, \sqrt{\sigma_m^2(\bfA)+\eta^2}}  \quad \mbox{if } m < n.$$
Notice that, $\bfD_\bfA$ is invertible if $\eta>0$ or $\rank{\bfA} = \max\{m,n\}$. By equation~\eqref{eq:eigendecompB} the left singular vectors of $\bfB$ correspond to the left singular vectors of $\bfA$, i.e., $\bfU_\bfB = \bfU_\bfA$. As for the right singular vectors let 
$$\bfV_\bfB = \begin{bmatrix} \bfV_{11} & 	\bfV_{12} \\ 	\bfV_{21} & 	\bfV_{22} \\ \end{bmatrix}$$
with $\bfV_{11} \in \bbR^{n \times m}, \bfV_{21} \in \bbR^{m \times m}, \bfV_{12} \in \bbR^{n \times n}$, and $\bfV_{22} \in \bbR^{m \times n}$. Then
$$ \bfB =  [\bfA \ \ \eta \, \bfI_m] = \bfU_\bfA \left[ \bfD_\bfA \ \  \bfzero_{m \times n} \right] \begin{bmatrix} \bfV_{11}\t & 	\bfV_{21}\t \\ 	\bfV_{12}\t & 	\bfV_{22}\t \\ \end{bmatrix} = [\bfU_\bfA \bfD\bfD_\bfA \bfV_{11}\t \ \ \ \bfU_\bfA \bfD_\bfA \bfV_{21}\t m]$$
and $ \bfV_{11}= \bfV_\bfA \bfSigma_\bfA\t\bfD_\bfA^{-1}$ and $\bfV_{21} = \eta \, \bfU_\bfA\bfD_\bfA^{-1}$. The matrices $\bfV_{12}$ and $\bfV_{22}$ are any matrices satisfying
$\bfV_{12}\t\bfV_{12}+\bfV_{22}\t\bfV_{22} = \bfI_n$ and $\bfV_{11}\t\bfV_{12}+\bfV_{21}\t\bfV_{22} = \bfzero_{m\times n}$ or equivalently $\bfA\bfV_{12}+\eta\bfV_{22} = \bfzero_{m\times n}$.
\end{proof}

Next, we provide a main result of our paper.	

\begin{theorem}\label{thm:mainresult}
Given matrices $\bfA \in \bbR^{m \times n}$, $\bfM \in \bbR^{n \times p}$, and $\bfP \in \bbR^{n \times m}$, with $\rank{\bfA} = k \leq n \leq m$, $ \rank{\bfM} = n \leq p$, let index $r \leq k$ and parameter $\eta \geq 0$, nonzero if $r < m$.  Define $\bfF = (\bfI_n - \bfP \bfA)\bfM\bfM\t\bfA\t - \eta^2 \bfP$. If $\rank{\bfF} \geq r$, then a global minimizer $\widehat \bfZ \in \bbR^{n \times m}$ of the problem
\begin{equation}
	\label{eqn:thmproblem}
	\min_{\rank{\bfZ} \leq r} \,\,f(\bfZ) = \norm[\fro]{ \bfZ \begin{bmatrix}
	\bfA \bfM & \eta \bfI_m \end{bmatrix} - \begin{bmatrix}
		\bfM - \bfP \bfA \bfM & -\eta \bfP
	\end{bmatrix}}^2
\end{equation}
is given by
\begin{equation} \label{eq:zhat}
	\widehat \bfZ = \bfU_{\bfH,r}\bfU_{\bfH,r}\t \bfF (\bfA \bfM \bfM\t \bfA\t +\eta^2 \bfI)^{-1},
\end{equation}
where symmetric matrix $\bfH = \bfF (\bfA \bfM \bfM\t \bfA\t +\eta^2 \bfI)^{-1} \bfF\t$ has eigenvalue decomposition $\bfH = \bfU_\bfH \bfLambda_\bfH \bfU_\bfH\t$ with eigenvalues ordered so that $\lambda_j \geq \lambda_i$ for $j < i \leq n$, and $\bfU_{\bfH,r}$ contains the first $r$ columns of $\bfU_{\bfH}$. 
Moreover, $\widehat \bfZ$ is the unique global minimizer of \eqref{eqn:thmproblem} if and only if $\lambda_r > \lambda_{r+1}$.
\end{theorem}

\begin{proof}
	We will use Theorem~\ref{thm:Friedland} where $\bfB = \left[ \left( \bfI_n - \bfP\bfA \right) \bfM \, \ \ \,  -\eta \bfP \right] $ and $\bfC = \left[ \bfA\bfM \, \ \ \, \eta \bfI_m \right]$.
Let 	
$$ \bfU\t\bfA \bfG = \bfSigma \quad \mbox{and} \quad \bfV\t\bfM\t\bfG = \bfS$$
with	
$$ \bfSigma = \begin{bmatrix} \diag{\sigma_1,\ldots,\sigma_n} \\ \bfzero_{(m-n)\times n}\end{bmatrix} 
		\quad \mbox{and} \quad  
		\bfS = \begin{bmatrix} \diag{s_1,\ldots,s_n} \\ \bfzero_{(p-n)\times n}
	\end{bmatrix}$$ 
denote the generalized SVD of $\left\{ \bfA, \bfM\t \right\}$ and let $\bfL$ be defined by $\bfL = \bfSigma \bfG^{-1}\bfG^{-\top}\bfS\t$ with its SVD given by $ \bfL = \bfU_\bfL \bfSigma_\bfL \bfV_\bfL\t$. Then $\bfA\bfM = \bfU_{\bfA\bfM} \bfSigma_\bfL \bfV_{\bfA\bfM}\t$, where $\bfU_{\bfA\bfM}  = \bfU \bfU_\bfL$ and $\bfV_{\bfA\bfM}  = \bfV \bfV_\bfL$. Using Lemma~\ref{lem:svdExtended}, the SVD of $\bfC$ is given by
$$ 
\bfU_\bfC = \bfU_{\bfA\bfM}, \quad \bfSigma_\bfC = 
	\begin{bmatrix}
		\bfD_{\bfA\bfM} & \bfzero_{m \times p}
	\end{bmatrix}
\quad \mbox{and} \quad 
\bfV_\bfC = 
	\begin{bmatrix}
		\bfV_{\bfA\bfM}\bfSigma_\bfL\t\bfD_{\bfA\bfM}^{-1} & \bfV_{12} \\[1ex] 
		\eta \, \bfU_{\bfA\bfM}\bfD_{\bfA\bfM}^{-1}         & \bfV_{22}
	\end{bmatrix}, 
$$
with 
\begin{align*}
		\bfD_{\bfA\bfM} &= \diag{\sqrt{\sigma_1^2(\bfA\bfM)+\eta^2},\ldots, \sqrt{\sigma_n^2(\bfA\bfM)+\eta^2}, \eta, \ldots, \eta}, \quad\mbox{for } m\geq p, \\
		% \mbox{and}&\\
		\bfD_{\bfA\bfM} &= \diag{\sqrt{\sigma_1^2(\bfA\bfM)+\eta^2},\ldots, \sqrt{\sigma_m^2(\bfA\bfM)+\eta^2}},\quad
		\mbox{for } m < p,
\end{align*}
and appropriately defined $\bfV_{12}$ and $\bfV_{22}$. Notice that $\bfD_{\bfA\bfM}$ is invertible and $\rank{\bfC} = m$, if either $\eta >0$ or $\rank{\bfA\bfM} = m$.  Also acknowledge that $\bfD_{\bfA\bfM}^2 =\bfSigma_\bfL\bfSigma_\bfL\t+\eta^2 \bfI_m$. Thus, the pseudoinverse of $\bfC$ is given by 
$$
\bfC^\dagger =
	\begin{bmatrix}
		\bfV_{\bfA\bfM}      & \bfzero_{p\times m} \\
		\bfzero_{m \times p} & \bfU_{\bfA\bfM}
	\end{bmatrix}
	\begin{bmatrix}
		\bfSigma_\bfL\t \\
		\eta \, \bfI_m
	\end{bmatrix}
	\bfD_{\bfA\bfM}^{-2}\bfU_{\bfA\bfM}\t
$$
and
$$
\bfV_{\bfC,m}\bfV_{\bfC,m}\t =
\begin{bmatrix}
	\bfV_{\bfA\bfM}\bfSigma_\bfL\t\bfD_{\bfA\bfM}^{-2}\bfSigma_\bfL\bfV_{\bfA\bfM}\t & 
	\eta \, \bfV_{\bfA\bfM}\bfSigma_\bfL\t\bfD_{\bfA\bfM}^{-2}
	\bfU_{\bfA\bfM}\t\\[1ex] 
	\eta \, \bfU_{\bfA\bfM}\bfD_{\bfA\bfM}^{-2}\bfSigma_\bfL\bfV_{\bfA\bfM}\t &
	\eta^2 \, \bfU_{\bfA\bfM}\bfD_{\bfA\bfM}^{-2} \bfU_{\bfA\bfM}\t
\end{bmatrix}.
$$
Let $\bfF = (\bfI_n - \bfP \bfA)\bfM \bfV_{\bfA\bfM}\bfSigma_\bfL\t \bfU_{\bfA\bfM}\t -\eta^2 \, \bfP $, then 
\begin{align}
	\label{eqn:defineK}
\bfK &= \bfB\bfV_{\bfC,m}\bfV_{\bfC,m}\t = \bfF \bfU_{\bfA \bfM}\bfD_{\bfA\bfM}^{-2}
	\begin{bmatrix}  \bfSigma_\bfL\bfV_{\bfA\bfM}\t	& \eta \, \bfU_{\bfA\bfM}\t 
	\end{bmatrix}.
\end{align}

Notice that $\rank{\bfK} \geq r$, since $\rank{\bfF} \geq r$ by assumption.
Then, let symmetric matrix $\bfH = \bfK \bfK\t = \bfF \bfU_{\bfA \bfM} \bfD_{\bfA\bfM}^{-2} \bfU_{\bfA \bfM}\t \bfF\t$ have eigenvalue decomposition $\bfH = \bfU_\bfH \bfLambda_\bfH \bfU_\bfH\t$ with eigenvalues ordered so that $\lambda_j \geq \lambda_i,$ for $j < i \leq n$.  Next we proceed to get an SVD of $\bfK$,
$$ \bfK = \bfU_\bfH \left[ \bfLambda_\bfH^{1/2}\, |\, \bfzero_{n \times(m+p-n)} \right] \bfV_\bfK\t $$
with	$$ \bfV_\bfK = \begin{bmatrix}
		\bfV_{11} &		\bfV_{12} & \bfV_{13}\\
		\bfV_{21} &		\bfV_{22}	 & \bfV_{23}
	\end{bmatrix}, $$
where $\bfV_{11} \in \bbR^{p \times r}, \bfV_{21} \in \bbR^{m \times r}, \bfV_{12} \in \bbR^{p \times(n-r)},$ and remaining matrices are defined accordingly.
Then equating the SVD of $\bfK$ with~\eqref{eqn:defineK} and using a similar argument as in Lemma~\ref{lem:svdExtended}, we get 
$$\bfU_\bfH\t \bfF \bfU_{\bfA \bfM} \bfD_{\bfA\bfM}^{-2}\bfSigma_\bfL\bfV_{\bfA\bfM}\t  = \bfLambda_\bfH^{1/2} \begin{bmatrix} \bfV_{11}\t \\ \bfV_{12}\t 
	\end{bmatrix}$$
and
$$ \eta \bfU_\bfH\t \bfF \bfU_{\bfA \bfM} \bfD_{\bfA\bfM}^{-2} \bfU_{\bfA\bfM}\t  = \bfLambda_\bfH^{1/2} \begin{bmatrix} \bfV_{21}\t \\ \bfV_{22}\t 
	\end{bmatrix}.$$
Since  $\bfLambda_{\bfH,r}$ (the principal $r \times r$ submatrix of $\bfLambda_\bfH$) is invertible, the transpose of the first $r$ columns of $\bfV_\bfK$ have the form,
\begin{align*}
	\bfV_{\bfK,r}\t &=	\left[  \bfV_{11}\t  \,| \,   \bfV_{21}\t               \right] \\
	  &=  \bfLambda_{\bfH,r}^{-1/2} 	\left[  \bfI_r       \,| \,   \bfzero_{r\times (n-r)} \right] \bfU_\bfH\t\bfF \bfU_{\bfA \bfM} \bfD_{\bfA\bfM}^{-2}    \left[\bfSigma_\bfL\bfV_{\bfA\bfM}\t \, | \, \eta \, \bfU_{\bfA\bfM}\t \right]\\
	  &=  \bfLambda_{\bfH,r}^{-1/2} \bfU_{\bfH,r}\t\bfF \bfU_{\bfA \bfM} \bfD_{\bfA\bfM}^{-2}    \left[\bfSigma_\bfL\bfV_{\bfA\bfM}\t \, | \, \eta \, \bfU_{\bfA\bfM}\t \right]
\end{align*}
and the best rank $r$ approximation of $\bfK$ is given by
\begin{align*}
	\bfK_r &= \bfU_{\bfH,r} \bfLambda_{\bfH,r}^{1/2}\bfV_{\bfK,r}\t \\
&=\bfU_{\bfH,r}\bfU_{\bfH,r}\t\bfF \bfU_{\bfA \bfM} \bfD_{\bfA\bfM}^{-2}    \left[\bfSigma_\bfL \, | \, \eta \, \bfI_m \right]
\begin{bmatrix}
	\bfV_{\bfA\bfM}\t & \bfzero_{p \times m} \\ \bfzero_{ m\times p} &\bfU_{\bfA\bfM}\t
\end{bmatrix}.
\end{align*}
Finally, using Theorem~\ref{thm:Friedland} we find that all global minimizers of $f$ with rank at most $r$ can be written as
\begin{align*}
	\widehat\bfZ &= \bfK_r \bfC^\dagger\\
			&=\bfU_{\bfH,r}\bfU_{\bfH,r}\t\bfF \bfU_{\bfA\bfM}\bfD_{\bfA\bfM}^{-2}  \left(\bfSigma_\bfL\bfSigma_\bfL\t +\eta^2\bfI_m\right)  \bfD_{\bfA\bfM}^{-2}\bfU_{\bfA\bfM}\t\\
&=\bfU_{\bfH,r}\bfU_{\bfH,r}\t\bfF (\bfA \bfM \bfM\t \bfA\t +\eta^2 \bfI)^{-1},
\end{align*}
where $\widehat \bfZ$ is a {\em unique} global minimizer of \eqref{eqn:thmproblem} if and only if $\lambda_r > \lambda_{r+1}$ since this condition makes the choice of $\bfU_{\bfH,r}$ unique.
\end{proof}

% section proof_general_case (end)

\section{Efficient methods to compute ORIM $\widehat\bfZ$} % (fold)
\label{sub:computational_methods_for_obtaining_bfz}

The computational cost to compute a global minimizer $\widehat \bfZ$ according to Theorem~\ref{thm:mainresult} requires the computation of a GSVD of $\left\{\bfA,\bfM\t\right\}$, an SVD of $\bfL$, and a partial eigenvalue decomposition of $\bfH$. For large-scale problems this may be computational prohibitive, so we seek an alternative approach to efficiently compute ORIM $\widehat \bfZ$. In the following we decompose the optimization problem into smaller subproblems and use efficient methods to solve the subproblems. The optimality of our update approach is verified by the following corollary of Theorem~\ref{thm:mainresult}.

\begin{corollary}\label{coro:thm}
Assume all conditions of Theorem~\ref{thm:mainresult} are fulfilled. Let $\widehat \bfZ_r$ be a global minimizer of~\eqref{eqn:thmproblem} of maximal rank $r$ and let $\widehat \bfZ_{r+\ell}$ be a global minimizer of~\eqref{eqn:thmproblem} of maximal rank $r+\ell$. Then $\tilde\bfZ_{\ell} = \widehat \bfZ_{r+\ell}-\widehat \bfZ_r$ is of maximal rank $\ell$ and the global minimizer of 
	\begin{equation}\label{eq:updateFormula}
		\tilde\bfZ_{\ell} = \argmin_{\rank{\bfZ}\leq \ell}\norm[\rm F]{\left(\widehat\bfZ_r+\bfZ \right)
		\begin{bmatrix} \bfA\bfM & \eta\,\bfI_m \end{bmatrix} - 
		\begin{bmatrix} \bfM-\bfP\bfA\bfM & -\eta\,\bfP \end{bmatrix}}^2 .
	\end{equation}
	Furthermore, $\tilde\bfZ_{\ell}$ is the unique global minimizer if and only if $\lambda_r > \lambda_{r+1}$ and $\lambda_{r+\ell} > \lambda_{r+\ell+1}$.
\end{corollary}

The significance of the corollary is as follows. Assume we are given a rank $r$ approximation $\widehat\bfZ_r$ and we are interested in updating our approximation to a rank $r+\ell$ approximation $\widehat \bfZ_{r+\ell}$. To calculate the optimal rank $r+\ell$ approximation $\widehat \bfZ_{r+\ell}$, we just need to solve a rank $\ell$ optimization problem of the form~\eqref{eq:updateFormula} and then update the solution, $\widehat\bfZ_{r+\ell} = \widehat \bfZ_r + \tilde\bfZ_{\ell}$. Thus, computing a rank $r$ ORIM matrix $\widehat \bfZ_r$ can be achieved by solving a sequence of smaller rank problems and updating the solutions. Algorithm~\ref{alg:rank1update} describes such an rank-1 update approach.
\begin{algorithm} 
\caption{(rank-1 update approach)}\label{alg:rank1update}
\begin{algorithmic}[1]
\REQUIRE $\bfA,\bfM, \bfP, \eta$
\STATE set $\widehat\bfZ_0 = \bfzero_{n\times m}$, $r = 0$
\WHILE{stopping criteria not reached}
\STATE 	$\displaystyle
		\tilde\bfZ_r = \argmin_{\rank{\bfZ}\leq 1}\norm[\rm F]{\left(\widehat\bfZ_r+\bfZ \right)
		\begin{bmatrix} \bfA\bfM & \eta\,\bfI_m \end{bmatrix} - 
		\begin{bmatrix} \bfM-\bfP\bfA\bfM & -\eta\,\bfP \end{bmatrix}}^2$ \label{alg:lineOpt}
\STATE $\widehat\bfZ_{r+1} = \widehat\bfZ_{r} + \tilde\bfZ_r$
\STATE $r = r+1$
\ENDWHILE
\ENSURE optimal $\widehat\bfZ_r$
\end{algorithmic}
\end{algorithm}

The main question in Algorithm~\ref{alg:rank1update} is how to efficiently solve the optimization problem in line~\ref{alg:lineOpt}. First, we reformulate the rank-1 constraint by letting $\bfZ = \bfx\bfy\t,$ where $\bfx\in\bbR^n$ and $\bfy\in\bbR^m$ and defining $\bfX_r = [\bfx_1,\ldots,\bfx_r] \in \bbR^{n \times r}$ and $\bfY_r = [\bfy_1,\ldots,\bfy_r] \in \bbR^{m \times r}$.  Then $\widehat\bfZ_r = \bfX_r\bfY_r\t$, and the optimization problem in line~\ref{alg:lineOpt} of Algorithm~\ref{alg:rank1update} reads
\begin{equation}\label{eq:optxy}
	\resizebox{.92 \textwidth}{!} 
	{$\displaystyle
	(\bfx_{r+1},\bfy_{r+1}) = \argmin_{(\bfx,\bfy)}
			\norm[\rm F]{\left(\bfX_r\bfY_r\t+\bfx\bfy\t \right)
			\begin{bmatrix} \bfA\bfM & \eta\,\bfI_m \end{bmatrix} - 
			\begin{bmatrix} \bfM-\bfP\bfA\bfM & -\eta\,\bfP \end{bmatrix}}^2.
	$}			
\end{equation}
Although standard optimization methods could be used, care must be taken since
this quartic problem is of dimension $n+m$ and ill-posed since the decomposition $\bfZ = \bfx\bfy\t$ is not unique. Notice that for fixed $\bfy$, optimization problem~\eqref{eq:optxy} is quadratic and convex in $\bfx$ and vise versa.  Thus, we propose to use an alternating direction optimization approach.  Assume $\bfx \neq \bfzero_{n\times 1}$, $\bfy \neq \bfzero_{m \times 1}$, and $\eta>0$, then the partial optimization problems resulting from~\eqref{eq:optxy} are ensured to have unique minimizers
\begin{equation}\label{eq:minx}
		\widehat\bfx
	 = \frac{\bfM\bfM\t\bfA\t\bfy -(\bfP+\bfX_r\bfY_r\t)\left(\bfA\bfM\bfM\t\bfA\t + \eta^2\bfI_m\right)\bfy} 
	 {\bfy\t\left( \bfA\bfM\bfM\t\bfA\t+\eta^2\bfI_m\right)\bfy} \quad \mbox{ for fixed } \bfy,
\end{equation}
and
$$
	\widehat\bfy = \frac{\left( \bfA\bfM\bfM\t\bfA\t+\eta^2\bfI_m \right)^{-1}\bfA\bfM\bfM\t\bfx  - (\bfP + \bfX_r\bfY_r\t)\t \bfx }{\bfx\t\bfx} \quad \mbox{ for fixed } \bfx.
$$
Notice that computing $\widehat\bfx$ in~\eqref{eq:minx} only requires matrix-vector products, while
computing $\widehat \bfy$ requires a linear solve. Since decomposition $\bfZ = \bfx\bfy\t$ is not unique, we propose to select the computationally convenient decomposition where $\norm[2]{\bfx} = 1$ and $\bfx \perp \bfX_r$. This results in a simplified formula for $\widehat \bfy$, i.e.,
\begin{equation}\label{eq:miny}
		\widehat\bfy = \left( \bfA\bfM\bfM\t\bfA\t+\eta^2\bfI_m \right)^{-1}\bfA\bfM\bfM\t\bfx  - \bfP\t \bfx.
\end{equation}
Noticing that~\eqref{eq:miny} is just the normal equations solution to the following least squares problem,
\begin{equation}\label{eq:minyls}
\min_\bfy	\norm[2]{\begin{bmatrix} \bfM\t\bfA\t	\\ \eta\, \bfI_m \end{bmatrix} \bfy 	- \begin{bmatrix}
		\bfM\t\bfx - \bfM\t\bfA\t\bfP\t\bfx  \\  -\eta \bfP\t\bfx
	\end{bmatrix}}\,,
\end{equation}
we propose to use a computationally efficient least squares solver such as LSQR \cite{PaSa82a,PaSa82b}, where various methods can be used to exploit the fact that the coefficient matrix remains constant \cite{Chen2005,Benzi2002}. In addition, quasi Newton methods may improve efficiency by taking advantage of a good initial guess and a good approximation on the inverse Hessian \cite{Nocedal1999}, but such comparisons are beyond the scope of this paper.

The alternating direction approach to compute a rank-1 update is provided in Algorithm~\ref{alg:alternating}.
\begin{algorithm} 
\caption{(alternating direction approach to compute rank-1 update)}\label{alg:alternating}
\begin{algorithmic}[1]
\REQUIRE $\bfA, \bfM,\eta,\bfZ, \bfP, r$
\STATE set $\widehat\bfy = \bfone_{m \times 1}$
\WHILE{stopping criteria not reached}
	\STATE get $\widehat \bfx$ by~\eqref{eq:minx}
	\STATE normalize $\widehat\bfx = \widehat \bfx/ \norm[2]{\widehat \bfx}$
	\STATE orthogonalize by $\widehat\bfx = \widehat\bfx - \bfX_{r}\bfX_{r}\t\widehat\bfx$
	\STATE get $\widehat\bfy$ by solving~\eqref{eq:minyls}
\ENDWHILE
\STATE $\bfx_{r+1} = \widehat\bfx$ and $\bfy_{r+1} =\widehat\bfy$
\ENSURE optimal $\bfx_{r+1}$ and $\bfy_{r+1}$
\end{algorithmic}
\end{algorithm}

In summary, our proposed method to compute low-rank ORIM $\widehat\bfZ$ combines Algorithms~\ref{alg:rank1update} and~\ref{alg:alternating}.  An efficient {\sc Matlab} implementation can be found at the following website:
		\begin{center}
			\texttt{\url{https://github.com/juliannechung/ORIM.git}}
		\end{center}
		\vspace*{1ex}
Before providing illustrations and examples of our method, we make a few remarks regarding numerical implementation.
\begin{enumerate}
	\item \emph{Storage}. Algorithmically $\widehat\bfZ_r$ need never be constructed, as we only require matrices $\bfX_r$ and $\bfY_r$. This decomposition is storage preserving as long as $r\leq \frac{mn}{m+n}$ and is ideal for problems where $\bfZ$ is too large to compute or $\bfA$ can only be accessed via function call.
	\item \emph{Stopping criteria}.  For Algorithm~\ref{alg:rank1update}, the specific rank $r$ for $\widehat\bfZ_{r}$ may be user-defined, but
	oftentimes such information is not available a priori.
However, the rank-1 update approach allows us to track the improvement in the function value from rank~$r$ to rank~$r+1$.  Then an approximation of rank~$r$ is deemed sufficient when 
$f(\bfZ_{r-1}) - f(\bfZ_r) < {\rm tol}\cdot f(\bfZ_r)$, where our default tolerance is ${\rm tol} = 10^{-6}$.
Standard stopping criteria \cite{Gill1981} can be used for Algorithm~\ref{alg:alternating}.  In particular, we track improvement in the function values $f(\bfX_r\bfY_r\t)$, track changes in the arguments $\widehat\bfx$ and $\widehat\bfy$, and set a maximum iteration. Our default tolerance is $10^{-6}$.
	\item \label{item:fro} \emph{Efficient function evaluations.} Rather than computing the function value $f(\bfX_r\bfY_r\t)$ from scratch at each iteration (e.g., for determining stopping criteria), efficient updates can be done by observing that
\begin{align*}
	f(\bfX_{r+1}\bfY_{r+1}\t) =& f(\bfX_r\bfY_r\t) \\
		&+ \bfy\t\left( \bfA\bfM\bfM\t\bfA\t + \eta^2\bfI_m\right)\left(\bfy+2\bfP\t\bfx\right) - 2 \bfy\t\bfA\bfM\bfM\t\bfx\,,
\end{align*}	 
where $f(\bfzero_{n \times m}) = \norm[\fro]{(\bfI_n-\bfP\bfA)\bfM}^2 + \eta^2\,\norm[\fro]{\bfP}^2$.  Since function evaluations are only relevant for the stopping criteria, they can be discarded, if desired, or approximated using trace estimators \cite{avron2011randomized}. 
	\item \emph{Initialization.} Equation~\eqref{eq:minx} requires an initial guess for $\bfy$. One uninformed choice may be $\bfy = \bfone_{m\times 1}$, and another option is to select $\bfy$ orthogonal to $\bfY_r$, i.e., $\bfy = (\bfI_m-\bfY_r\bfY_r\t)\bfr$ with $\bfr\in \bbR^m$ chosen at random.
	\item \emph{Symmetry.} If $\bfA$ and $\bfP$ are symmetric, our rank-1 update approach could be used to compute a symmetric ORIM $\widehat\bfZ_r = \bfX_r\bfX_r\t$, but the alternating direction approach should be replaced by an appropriate method for minimizing a quartic in $\bfx$.

	\item \emph{Covariance matrix.} Since $\bfM$ in our rank update approach only occurs in the product $\bfM\bfM\t$ and since $\bfM\bfM\t = \bfM_\bfxi \bfM_\bfxi\t + \bfmu_\bfxi \bfmu_\bfxi\t = \bfGamma_\bfxi + \bfmu_\bfxi \bfmu_\bfxi\t$, our algorithm can work directly with the covariance matrix.  Thus, a symmetric factorization does not need to be computed, which is important for various classes of covariance kernels \cite{saibaba2012application}.

\end{enumerate}

\section{Numerical Results} % (fold)
\label{sec:numerics}

In this section, we provide three experiments that not only highlight the benefits of ORIM updates but also demonstrate new approaches for solving inverse problems that use ORIM updates.  In Experiment 1, we use an inverse heat equation to investigate the efficiency and accuracy of our update approach. Then in Experiment 2, we use an image deblurring example to show that more accurate solutions to inverse problems can be achieved by using ORIM rank-updates to existing regularized inverse matrices. Lastly, in Experiment 3, we show that ORIM updates can be used in scenarios where perturbed inverse problems need to be solved efficiently and accurately.

\subsection{Experiment 1: Efficiency of ORIM rank update approach} % (fold)
\label{sub:experiment_1_efficiency_of_orim_vs_svd}
The goal of this example is to highlight our new result in Theorem~\ref{thm:mainresult} and to verify the accuracy and efficiency of the update approach described in Section~\ref{sub:computational_methods_for_obtaining_bfz}.  We consider a discretized (ill-posed) inverse heat equation derived from a Volterra integral equation of the first kind on $[0,1]$ with kernel $a(s,t) = k(s-t)$, where $k(t) = \frac{t^{-3/2}}{2 \sqrt{\pi}\kappa}\e^{-\frac{1}{4\kappa^2 t}}$. 
Coefficient matrix $\bfA$ is $1,\!000 \times 1,\!000$ and is significantly ill-posed for $\kappa \in [1,2]$.  We generate $\bfA$ using the \emph{Regularization Tools} package \cite{Hansen1994}. 

As a first study, we compare ORIM $\widehat\bfZ$ with other commonly used regularized inverse matrices.  Notice that $\widehat\bfZ$ is fully determined by $\bfA,\eta,\bfM$, and $\bfP$.

For this illustration, we select $\bfP$ and $\bfM$ to be realizations of random matrices whose entries are i.i.d.~standard normal $\calN(0,1)$, and we select $\kappa = 1$ and $\eta = 0.02$.  Then we compute ORIM $\widehat \bfZ$ as in Equation~\eqref{eq:zhat} for various ranks $r$ and plot the function values $f(\widehat \bfZ)$ in Figure~\ref{fig:Example1}. For comparison, we also provide function values for other commonly used rank-$r$ reconstruction matrices, including the TSVD matrix, $\bfA_{r}^\dagger,$ the truncated Tikhonov matrix~\eqref{eqn:TTik} (TTik), and the matrix provided from Theorem~1 of \cite{Chung2015}, here referred to as ORIM$_0$. Notice that TTik and ORIM$_0$ matrices are just special cases of ORIM where $\bfM = [\,\bfI_n \ \ \bfzero_{n \times 1}\,]$ and $\bfP = \bfzero_{n \times m}$ for TTik and $\bfM = [\,\bfM_\bfxi \ \ \bfzero_{n \times 1}\,]$ and $\bfP = \bfzero_{n \times m}$ for ORIM$_0$. Figure~\ref{fig:Example1} shows that, as expected, the function values for ORIM are smallest for all computed ranks.
\begin{figure}[bthp!]
	\begin{center}
		\includegraphics[width=0.8\textwidth]{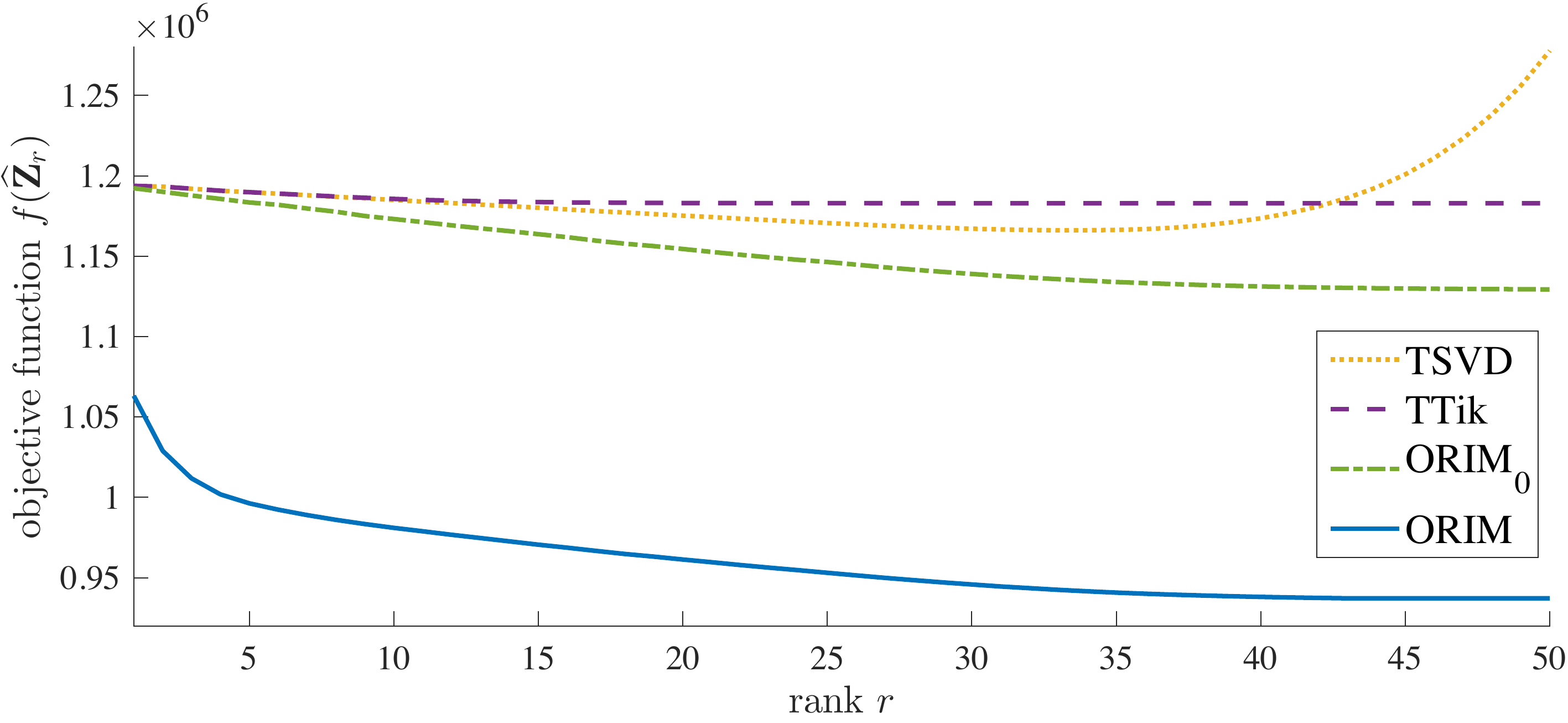}
	\end{center}
	\caption{Comparison of the function values $f(\bfZ)$ where $\bfZ$ corresponds to different reconstruction matrices. The dotted line refers to TSVD, the dashed line to truncated-Tikhonov, the dash-dotted line to ORIM$_0$ (i.e., ORIM where $\bfM_\bfxi = \bfI_n$ and $\bfmu_\bfxi=\bfzero_{n\times 1}$), and the solid line to ORIM $\widehat \bfZ$.  Results correspond to a discretized Volterra integral equation.}
	\label{fig:Example1}
\end{figure}  

We also verified our proposed rank-update approach by comparing function values computed with the rank update approach to those from Theorem~\ref{thm:mainresult}. We observed that the relative absolute errors remained below $2.9485\cdot10^{-3}$ for all computed ranks $r$, making the plot of the function values for the update approach indistinguishable from the solid line in Figure~\ref{fig:Example1}.  Thus, we omit it for clarity of presentation.\\

Next, we illustrate the efficiency of our rank update approach for solving a sequence of ill-posed inverse problems. Such scenarios commonly occur in nonlinear optimization problems such as variable projection methods where nonlinear parameters are moderately changing during the optimization process \cite{Nocedal1999,GoPe73}.  Consider again the inverse heat equation, and assume that we are given a sequence of matrices $\bfA(\kappa_j) \in \bbR^{n\times n}$, where the matrices depend nonlinearly on parameter $\kappa_j$, and we are interested in solving a sequence of problems, $\bfb(\kappa_j)= \bfA(\kappa_j)\bfxi + \bfdelta_j$ for various $\kappa_j$. 

For each problem in the sequence, one could compute a Tikhonov solution $\bfxi_{\rm Tik}(\kappa_j) = \bfV_{\bfA(\kappa_j)}\bfPsi_{\bfA(\kappa_j)}\bfU_{\bfA(\kappa_j)}\t\bfb(\kappa_j)$, where 
$$\bfPsi_{\bfA(\kappa_j)} = \diag{\frac{\sigma_1(\bfA(\kappa_j))}{\sigma_1^2(\bfA(\kappa_j))+ \eta^2}, \ldots, \frac{\sigma_n(\bfA(\kappa_j))}{\sigma_n^2 (\bfA(\kappa_j))+ \eta^2}},$$
but this approach requires an SVD of $\bfA(\kappa_j)$ for each $\kappa_j$. We consider an alternate approach, where the SVD is computed once for a fixed $\kappa_j$ and then ORIM updates are used to obtain improved regularized inverse matrices for other $\kappa_j$'s. This approach relies on the fact that small perturbations in $\bfA(\kappa_j)$ lead to small rank updates in its inverse \cite{Stewart2001}.
	
Again for the inverse heat equation we use $n = 1,\!000$ and $\eta = 0.02$ and choose $\bfM = \bfI_n$ and $\bfmu = \bfzero_{n \times 1}$. We select equidistant values for $\kappa_j\in [1,2]$, $j = 1,\ldots, 100$, and let $\bfP^{(1)} = \bfV_{\bfA(\kappa_1)}\bfPsi_{\bfA(\kappa_1)}\bfU_{\bfA(\kappa_1)}\t$ be the Tikhonov reconstruction matrix corresponding to $\kappa_1$. Then for all other problems in the sequence, we compute reconstructions as
$$\bfxi_{\rm ORIM}(\kappa_{j+1}) = \bfP^{(j+1)}\bfb(\kappa_{j+1})$$
where $\bfP^{(j+1)} = \bfP^{(j)} + \bfX^{(j+1)}\left( \bfY^{(j+1)}\right)\t$,
where $\bfX^{(j+1)}$ and $\bfY^{(j+1)}$ are the low rank ORIM updates corresponding to $\bfA(\kappa_{j+1})$. We use a tolerance ${\rm tol} = 10^{-3}$. In Figure~\ref{fig:Ex2timing}, we report computational timings for the ORIM rank update approach, compared to the SVD, and in
Figure~\ref{fig:Ex2error} we provide corresponding relative reconstruction errors, computed as ${\rm rel} = \norm[2]{\bfxi_\star - \bfxi_{\rm true}} / \norm[2]{\bfxi_{\rm true}}$, where $\bfxi_\star$ is and approximation of $\bfxi$ (here, $\bfxi_{\rm ORIM}(\kappa_{j})$ and $\bfxi_{\rm Tik}(\kappa_{j})$). We observe that the ORIM update approach requires approximately half the required CPU time compared to the SVD, and the ORIM update approach can produce relative reconstruction errors that are comparable to and even slightly better than Tikhonov.  However, we also note potential disadvantages of our approach.  In particular, the SVD can be more efficient for small $n$, although ORIM updates are significantly faster for larger problems (results not shown).  Also, using different noise levels $\eta$ in each problem or taking larger changes in $\kappa_j$ may result in higher CPU times and/or higher reconstruction errors for the update approach.  We assume that the noise levels and problems are not changing significantly.  

\begin{figure}[h]
 \centering 
   \includegraphics[width=\textwidth]{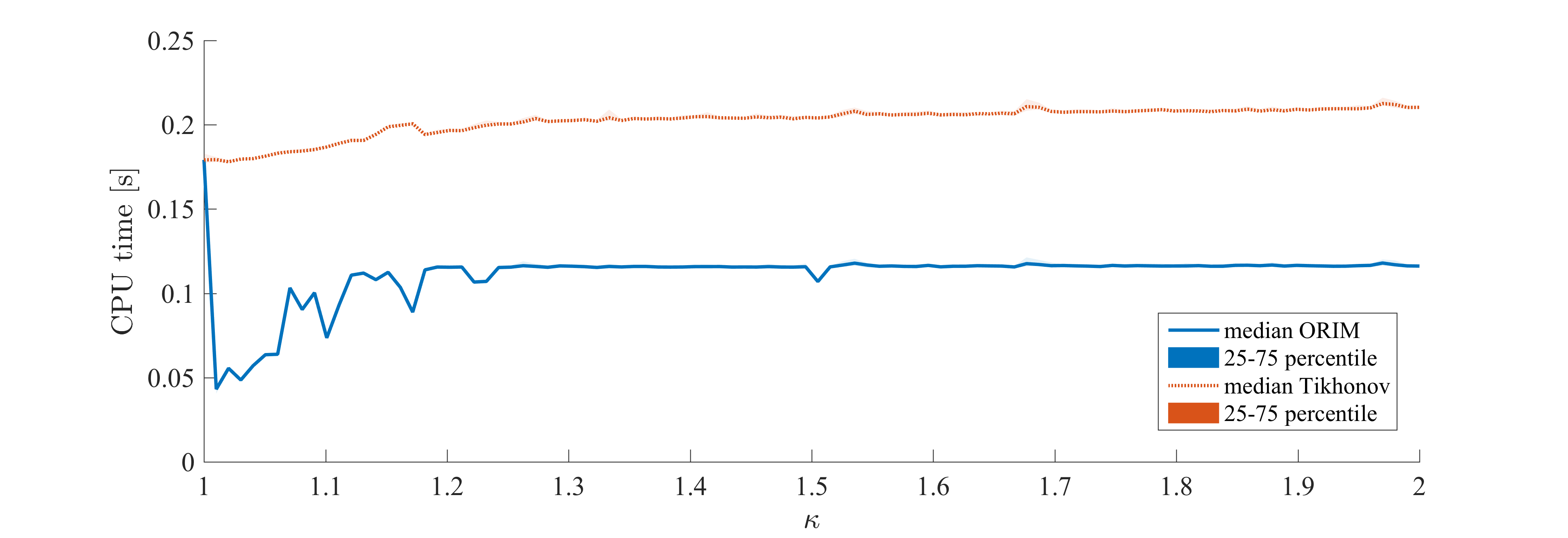}
 \caption{CPU times for computing a regularized inverse matrix using ORIM updates (solid line) and for computing the SVD to get a Tikhonov solution (dotted line) for a sequence of inverse problems varying in $\kappa$. We repeated the experiment 50 times and report the median as well as the 25-75th percentiles.}
 \label{fig:Ex2timing}
\end{figure}

\begin{figure}[h]
 \centering
   \includegraphics[width=\textwidth]{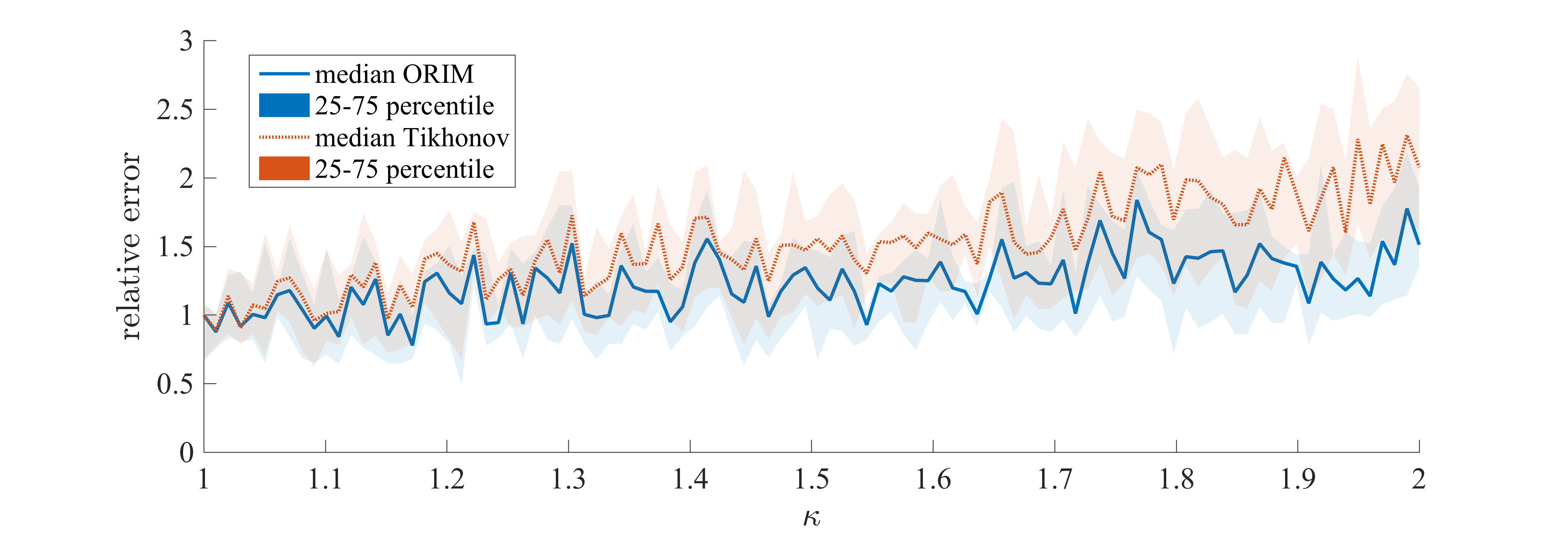}
 \caption{Relative reconstruction errors for reconstructions obtained using ORIM updates (solid line) and using Tikhonov regularization (dotted line). We report the median as well as the 25-75th percentiles for each $\kappa$ after repeating the experiment 50 times.}
 \label{fig:Ex2error}
\end{figure} 

% subsection experiment_1_efficiency_of_orim_vs_svd (end)

 \subsection{Experiment 2: ORIM Updates to Tikhonov} \label{sub:updating_inverse_matrices} % (fold)

Here we consider a classic image deblurring problem, where the model is given in~\eqref{eqn:linearsystem} where $\bfxi$ represents the desired image, $\bfA$ models the blurring process, and $\bfb$ is the blurred, observed image. The true image was taken to be the $15$-th slice of the 3D MRI image dataset that is provided in {\rm MATLAB}, which is $256 \times 256$ pixels. We assume spatially invariant blur, where the point spread function (PSF) is a $11 \times 11$ box-car blur.  We assume reflexive boundary conditions for the image.  Since the PSF is doubly symmetric, blur matrix $\bfA$ is highly structured and its singular value decomposition is given by $\bfA =\bfU_\bfA \bfSigma_\bfA \bfV_\bfA\t$, where here $\bfV_\bfA\t$ and $\bfU_\bfA$ represent the 2D discrete cosine transform (DCT) matrix and inverse 2D DCT matrix respectively \cite{Hansen2006}.  Here we use the RestoreTools package \cite{Nagy2004}.  Noise $\bfdelta$ was generated from a normal distribution, with zero mean, and scaled such that the noise level was $\norm[2]{\bfdelta}^2/ \norm[2]{\bfA \bfxi}^2= 0.01$. The true and observed images, along with the PSF, are provided in Figure~\ref{fig:mriproblem}. 

\begin{figure}[bthp!]
	\begin{center}
\begin{tabular}{ccc}
		\includegraphics[width=0.25\textwidth]{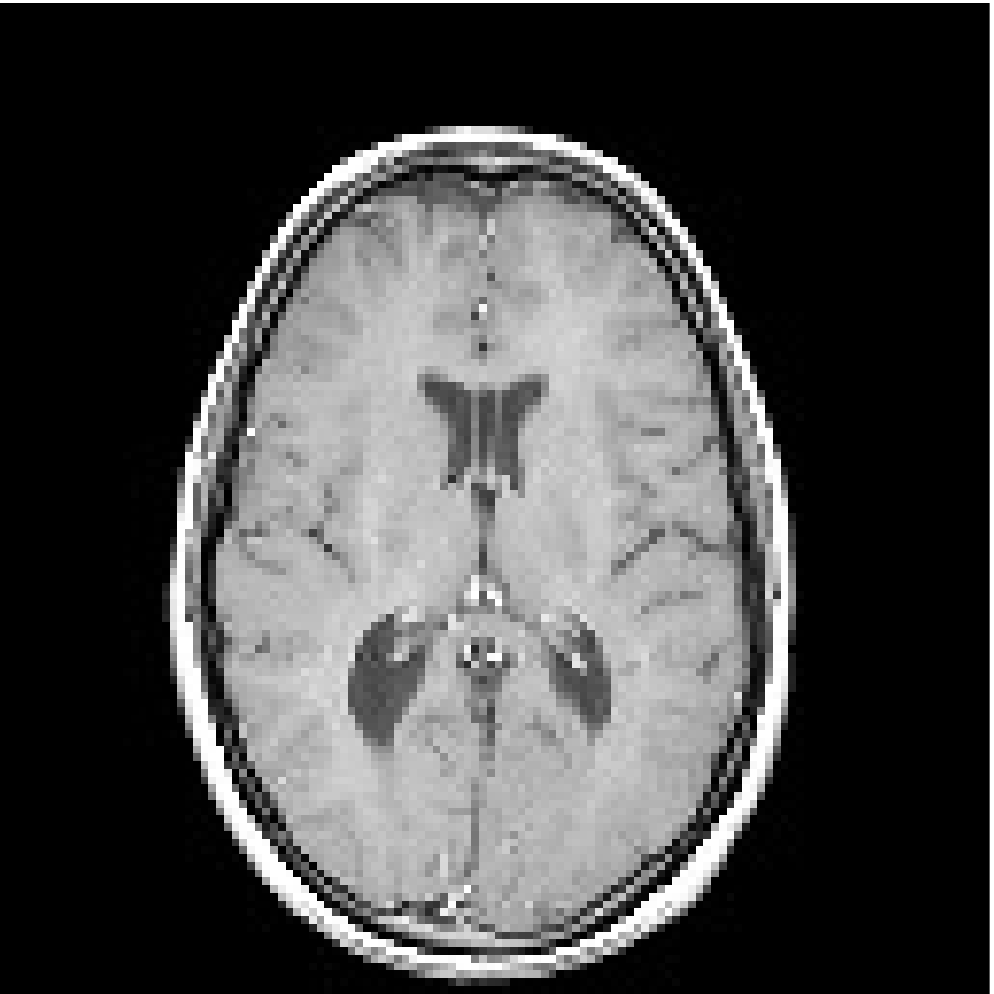} &
		\includegraphics[width=0.25\textwidth]{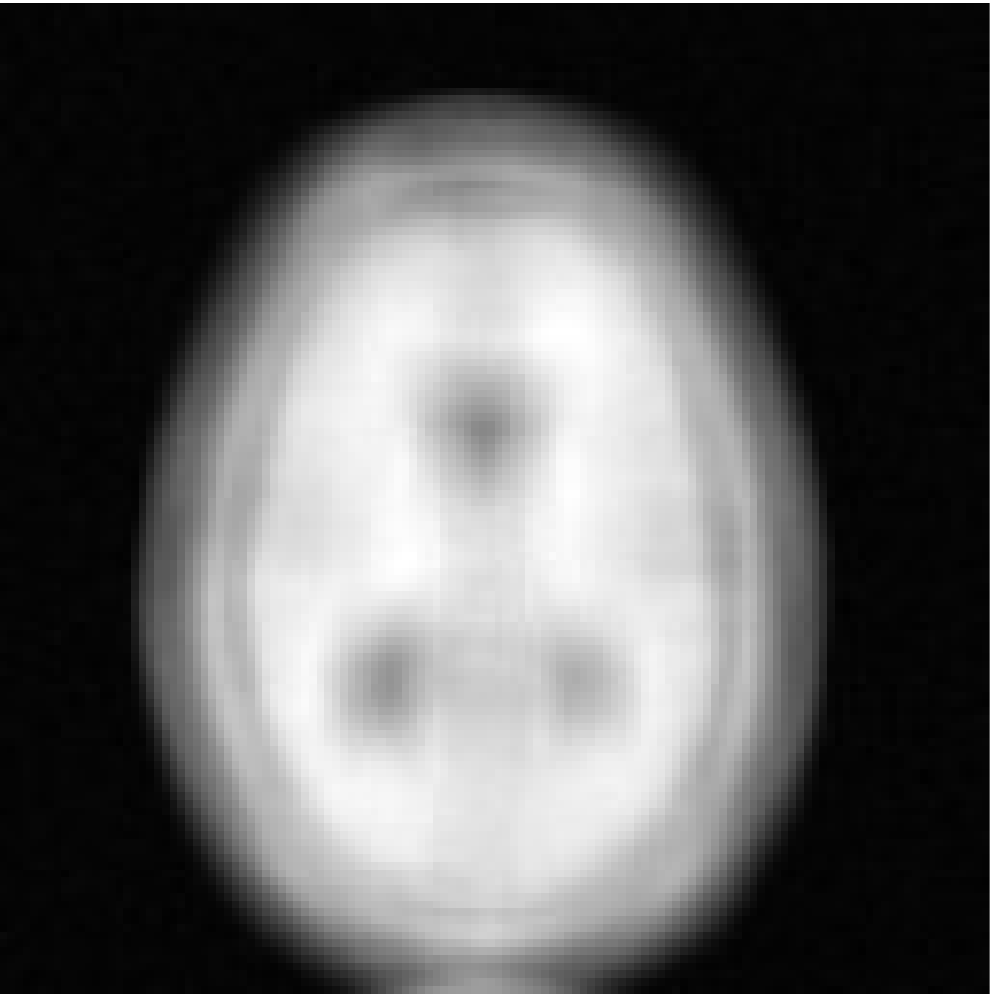}&
		\includegraphics[width=0.25\textwidth]{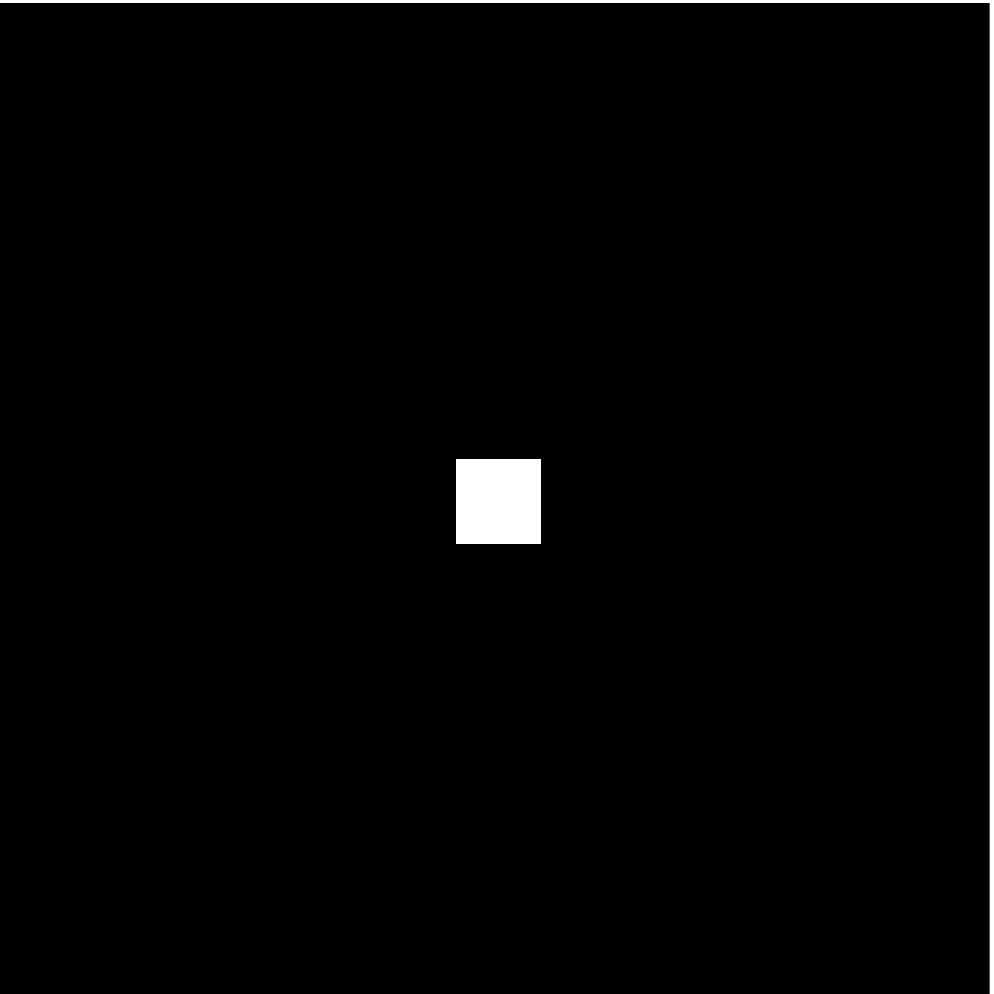}\\
		(a) True image & (b) Observed, blurred image & (c) Point spread function
\end{tabular}
	\end{center}
	\caption{Image deblurring example. The true (desired) MRI image is given (a). The observed, blurred image is provided in (b), and the PSF is provided in (c).}
	\label{fig:mriproblem}
\end{figure}

As an initial regularized inverse approximation, we use a Tikhonov reconstruction matrix, $\bfP = \bfV_\bfA (\bfSigma_\bfA\t\bfSigma_\bfA + \eta^2\bfI)^{-1} \bfSigma_\bfA^{-1} \bfU_\bfA\t$, where regularization parameter $\eta$ was selected to provide minimal reconstruction error.  That is, we used $\eta = 2.831\cdot10^{-2}$, which corresponded to the minimum of error function, $\norm[2]{\bfP\bfb-\bfxi}$. Although this approach uses the true image (which is not known in practice), our goal here is to demonstrate the improvement that can be obtained using the rank-update approach.  In practice, a standard regularization parameter selection method such as the generalized cross-validation could be used, which for this problem gave $\eta = 2.713\cdot10^{-2}$.  The Tikhonov reconstruction, $\bfP\bfb$, is provided in Figure~\ref{fig:recon}(a) along with the computed relative reconstruction error.

Next we consider various ORIM updates to $\bfP$ and evaluate corresponding reconstructions.  For the mean vector $\bfmu_\bfxi$, we use the image shown in Figure~\ref{fig:recon}(b), which was obtained by averaging images slices 8--22 of the MRI stack (omitting slice 15, the image of interest).  For efficient computations and simplicity, we assume $\bfGamma_\bfxi$ is diagonal with variances proportional to $\bfmu_\bfxi$, we choose, $\bfGamma_\bfxi = \diag{\bfmu_\bfxi}$; the matrix $\bfM_\bfxi$ is defined accordingly. We compute ORIM updates to $\bfP$ according to Algorithm~\ref{alg:rank1update} for the following cases of $\bfM$:
\begin{equation}
	\label{eqn:Mmatrices}
	\bfM_{(1)} = \begin{bmatrix}
		\bfI_n & \bfmu_\bfxi
	\end{bmatrix}, \quad \bfM_{(2)} = \begin{bmatrix}
		\bfM_\bfxi & \bfzero_{n \times 1}
	\end{bmatrix}, \quad \mbox{and} \quad
\bfM_{(3)} = \begin{bmatrix}
		\bfM_\bfxi & \bfmu_\bfxi
	\end{bmatrix}.
	\end{equation}
We refer to these matrix updates as $\widehat \bfZ_{(1)}$, $\widehat \bfZ_{(2)}$, and $\widehat \bfZ_{(3)}$ respectively, where $\widehat \bfZ_{(1)}$ is a rank-1 matrix and $\widehat \bfZ_{(2)}$ and $\widehat \bfZ_{(3)}$ are matrices of rank $5$.  Image reconstructions were obtained via matrix-vector multiplication,
$$\bfxi_{(j)} = \bfP\bfb + \widehat \bfZ_{(j)} \bfb, \quad \mbox{ for } j=1,2,3,$$
and are provided in Figure~\ref{fig:recon}(c)--(e).  Corresponding relative reconstruction errors are also provided.  
\begin{figure}[bthp!]
	\begin{center}
\begin{tabular}{cc}
		\includegraphics[width=0.25\textwidth]{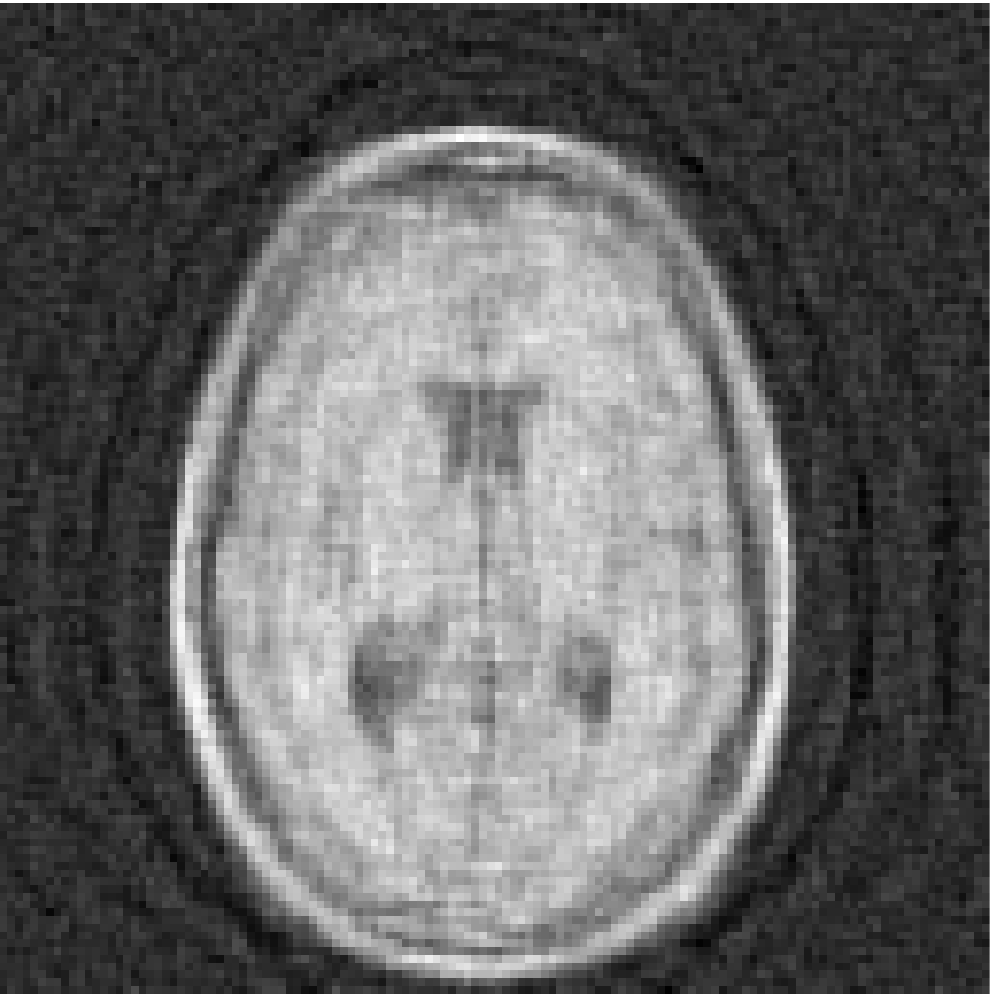}  &
		\includegraphics[width=0.25\textwidth]{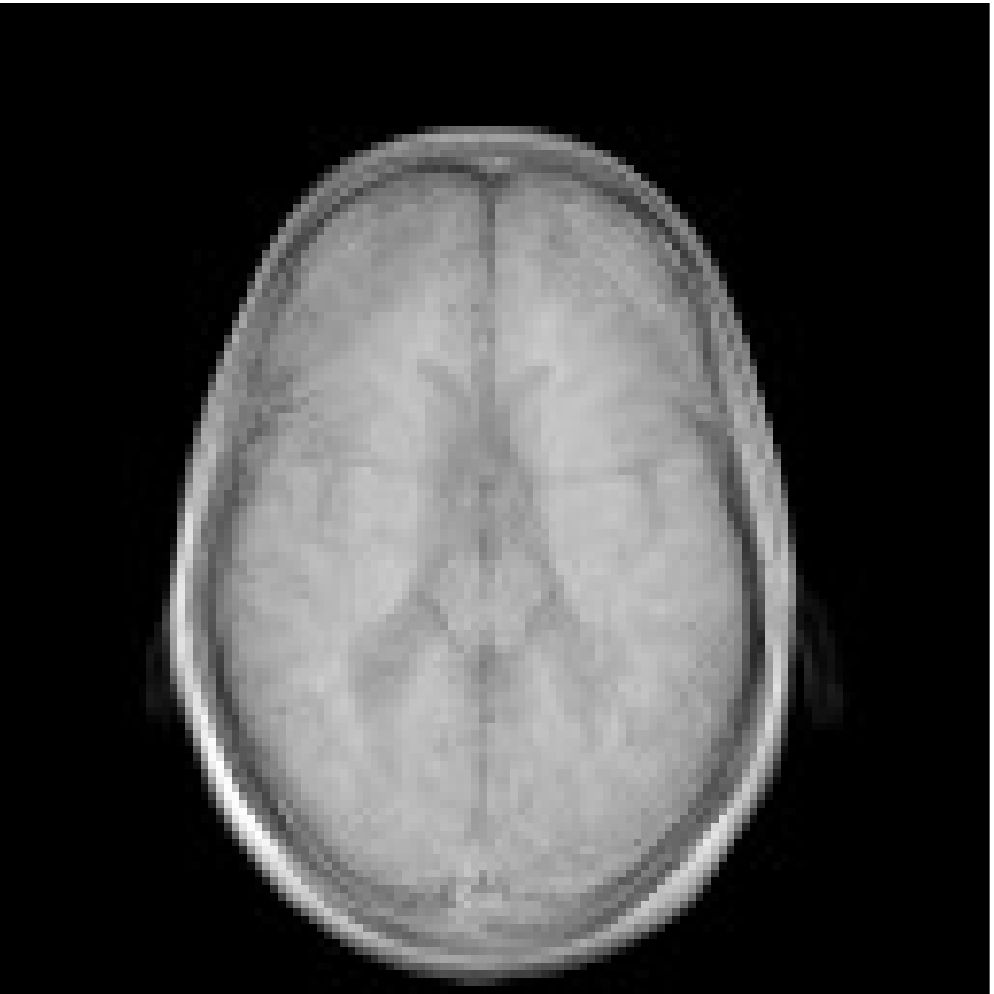}\\
		(a) Tikhonov, $\bfP\bfb$, $\rel = 0.2247$ & (b) Mean image, $\bfmu_\bfxi$\\[1ex]
\end{tabular}
\begin{tabular}{ccc}
		\includegraphics[width=0.25\textwidth]{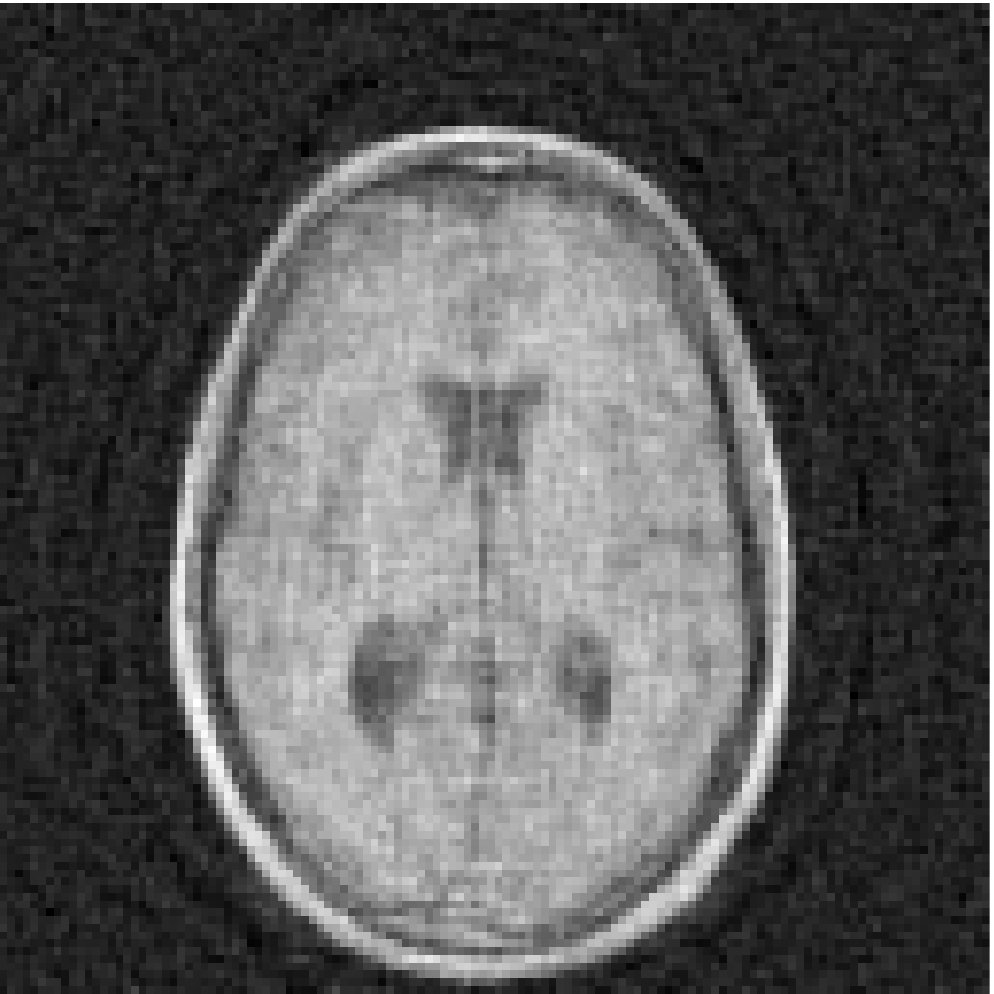}&
				\includegraphics[width=0.25\textwidth]{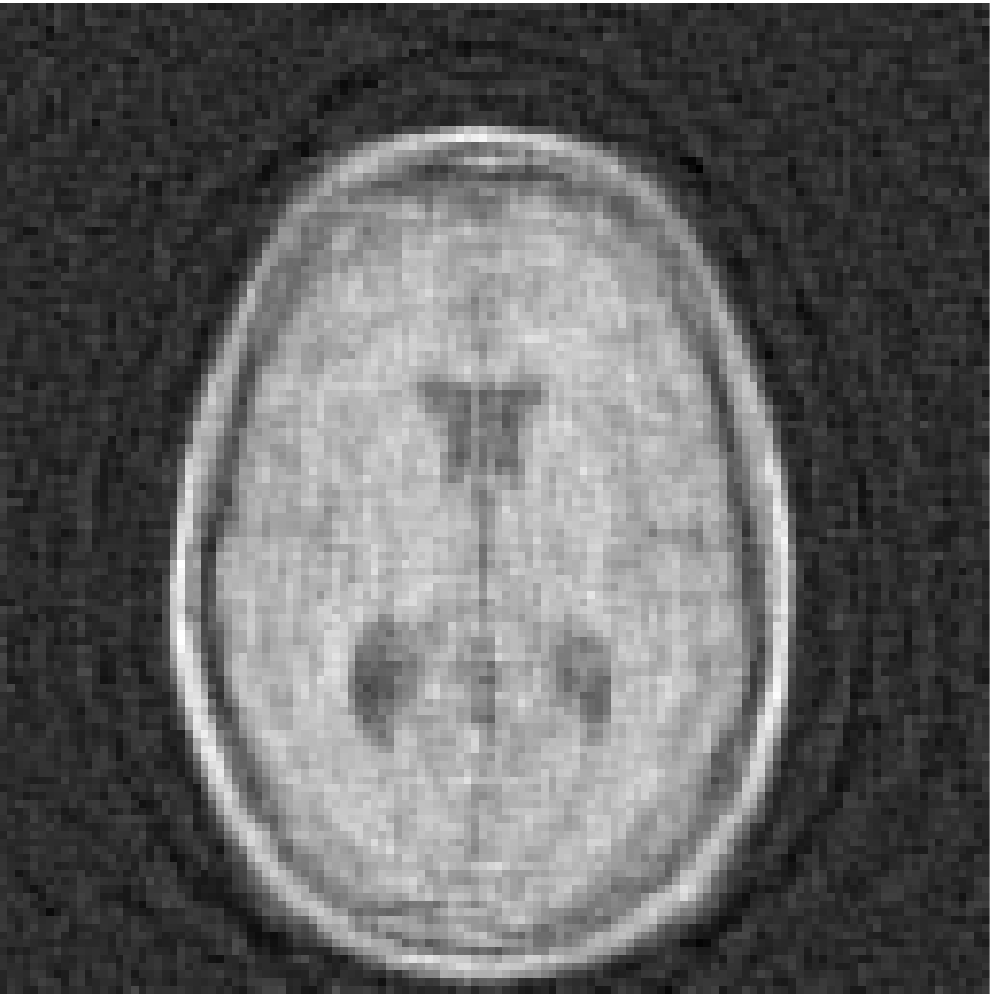}&
		\includegraphics[width=0.25\textwidth]{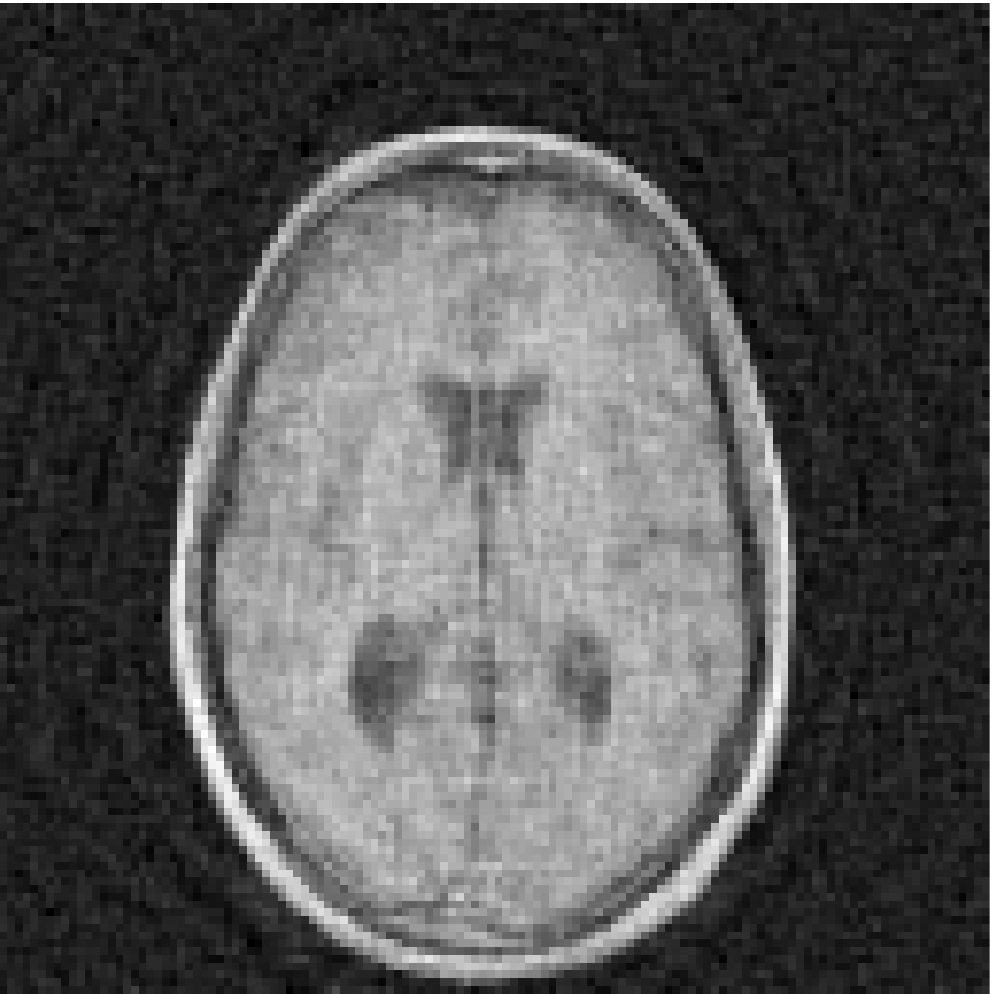}\\
		(c) $\bfxi_{(1)}$, $\rel = 0.1938$ & (d) $\bfxi_{(2)}$, $\rel = 0.2179$  &(e) $\bfxi_{(3)}$, $\rel = 0.1904$ 
\end{tabular}
	\end{center}
	\caption{Initial Tikhonov reconstruction is provided in (a).  The mean image, $\bfmu$, provided in (b), was taken to be the average of images slices 8-22 of the MRI image stack (omitting slice 15, the image of interest).  Image reconstructions in (c)-(e) correspond to ORIM updates to the initial Tikhonov reconstruction, for the various choices for $\bfM$ provided in~\eqref{eqn:Mmatrices}.  Relative reconstruction errors are provided.}
		\label{fig:recon}
\end{figure} 
Furthermore, absolute error images (in inverted colormap so that black corresponds to larger reconstruction error) in Figure~\ref{fig:errimages} show that the errors for the ORIM updated solution $\bfxi_{(3)}$ have smaller and more localized errors than the initial Tikhonov reconstruction.
\begin{figure}[bthp!]
	\begin{center}
		\begin{tabular}{cc}
		\includegraphics[width=.35\textwidth]{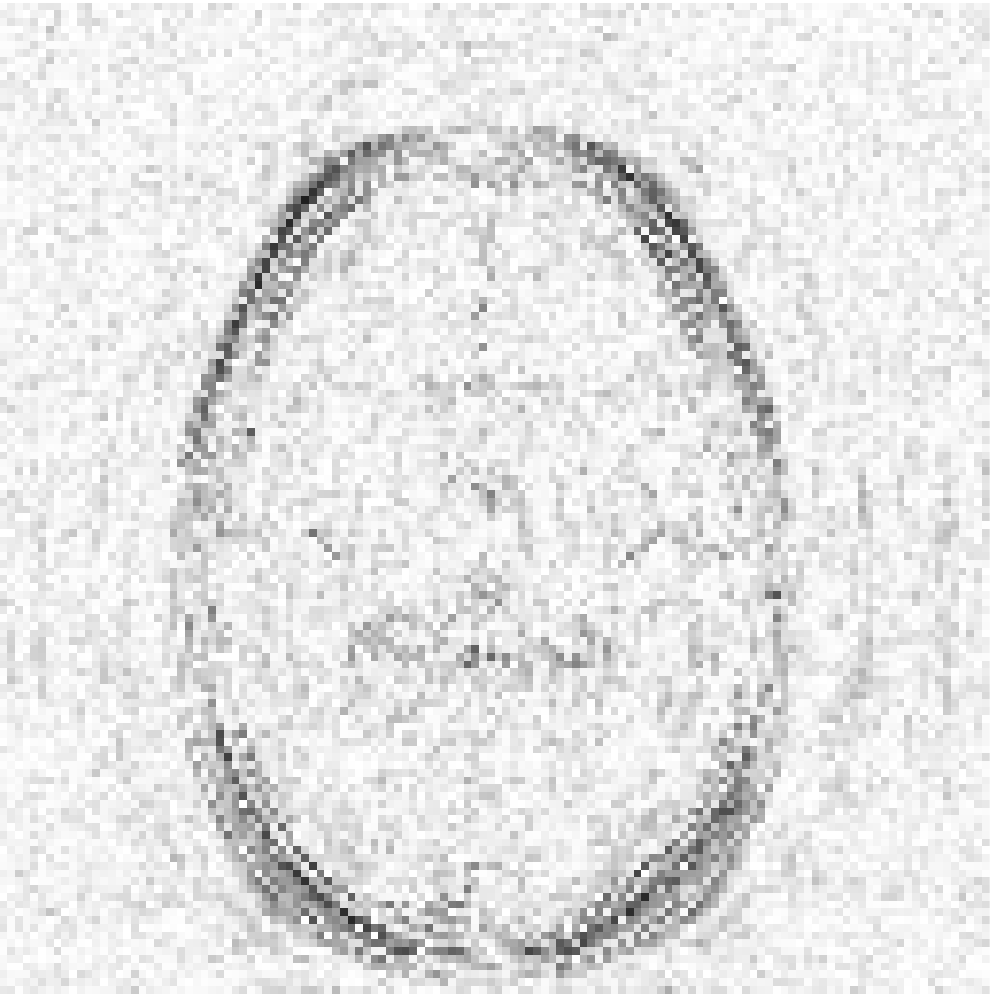} &
		\includegraphics[width=.35\textwidth]{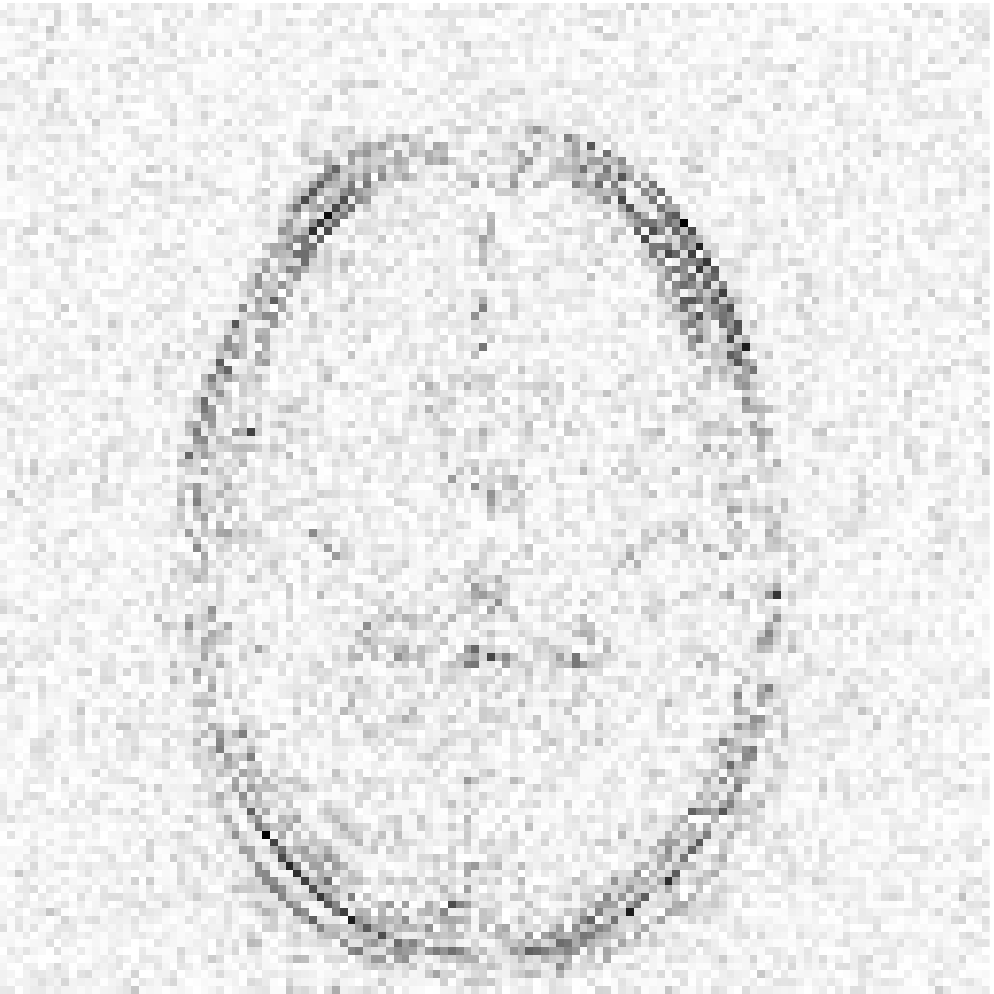}\\
		(a) Tikhonov & (b) $\bfxi_3$
		\end{tabular}
	\end{center}
	\caption{Error images (in inverted colormap where white corresponds to $0$) for the initial Tikhonov reconstruction and the ORIM updated solution $\bfxi_{(3)}$ which corresponds to $\bfM_{(3)}$ (i.e., nonzero mean and covariance matrix for $\bfxi$).}
		\label{fig:errimages}
\end{figure}

We repeated this experiment $20,\!000$ times, each time with a different noise realization in $\bfb$ and provide the distribution of the corresponding relative reconstruction errors in Figure~\ref{fig:hist}. Additionally, for each of these approaches, we provide the average reconstruction error, along with the standard deviation over all noise realizations in Table~\ref{tab:recon}.
It is evident from these experiments that ORIM rank-updates to the Tikhonov reconstruction matrix can lead to reconstructions with smaller relative errors and allows users to easily incorporate prior knowledge regarding the distributions of $\bfxi$ and $\bfdelta$. 
\begin{figure}[bthp!]
	\begin{center}
		\includegraphics[width=\textwidth]{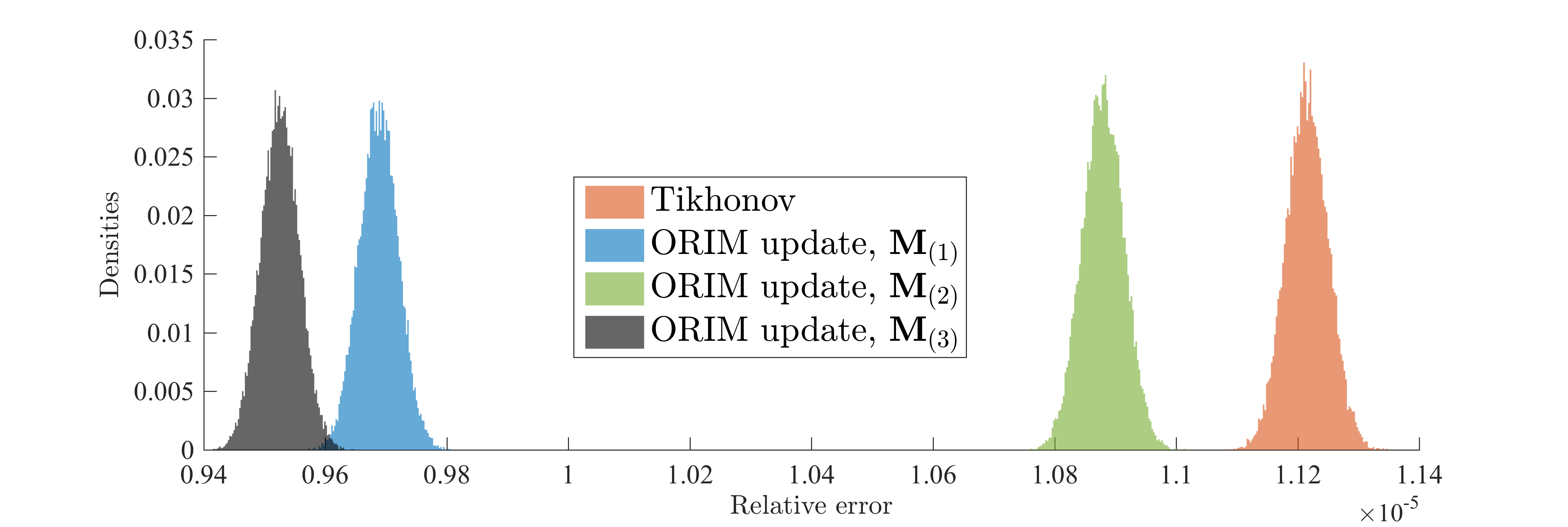}
	\end{center}
	\caption{Distributions of relative reconstruction errors}
		\label{fig:hist}
\end{figure}
\begin{table}[bthp]
	\label{tab:recon}
\caption{Comparison of average relative reconstruction error and standard deviation for $1,\!000$ noise realizations.}
	\begin{center}
\begin{tabular}{|l|c|}
	\hline
& mean \footnotesize $\pm$ standard deviation\\	\hline
 Tikhonov                  &  $1.1215\cdot 10^{-5}$ \ \footnotesize $\pm \, 3.4665\cdot 10^{-8}$ \\	\hline
 ORIM update, $\bfM_{(1)}$ &  $9.6881\cdot 10^{-6}$ \ \footnotesize $\pm \, 3.2040\cdot 10^{-8}$ \\	\hline
 ORIM update, $\bfM_{(2)}$ &  $1.0880\cdot 10^{-5}$ \ \footnotesize $\pm \, 3.4402\cdot 10^{-8}$ \\	\hline
 ORIM update, $\bfM_{(3)}$ &  $9.5254\cdot 10^{-6}$ \ \footnotesize $\pm \, 3.1541\cdot 10^{-8}$ \\	\hline
\end{tabular}
\end{center}
\end{table}
   
We then applied our reconstruction matrices, $\bfP + \widehat \bfZ_{(j)},$ to the other images in the MRI stack and provide the relative reconstruction errors in Figure~\ref{fig:mrislices}.  We observe that in general, all of the reconstruction matrices provide fairly good reconstructions, with smaller relative errors corresponding to images that are most similar to the mean image.  Some of the true images were indeed included in the mean image.  Regardless, our goal here is to illustrate that ORIM update matrices can be effective and efficient, if a good mean image and/or covariance matrix are provided. Other covariance matrices can be easily incorporated in this framework, but comparisons are beyond the scope of this work.
 
\begin{figure}[bthp!]
	\begin{center}
		\includegraphics[width=\textwidth]{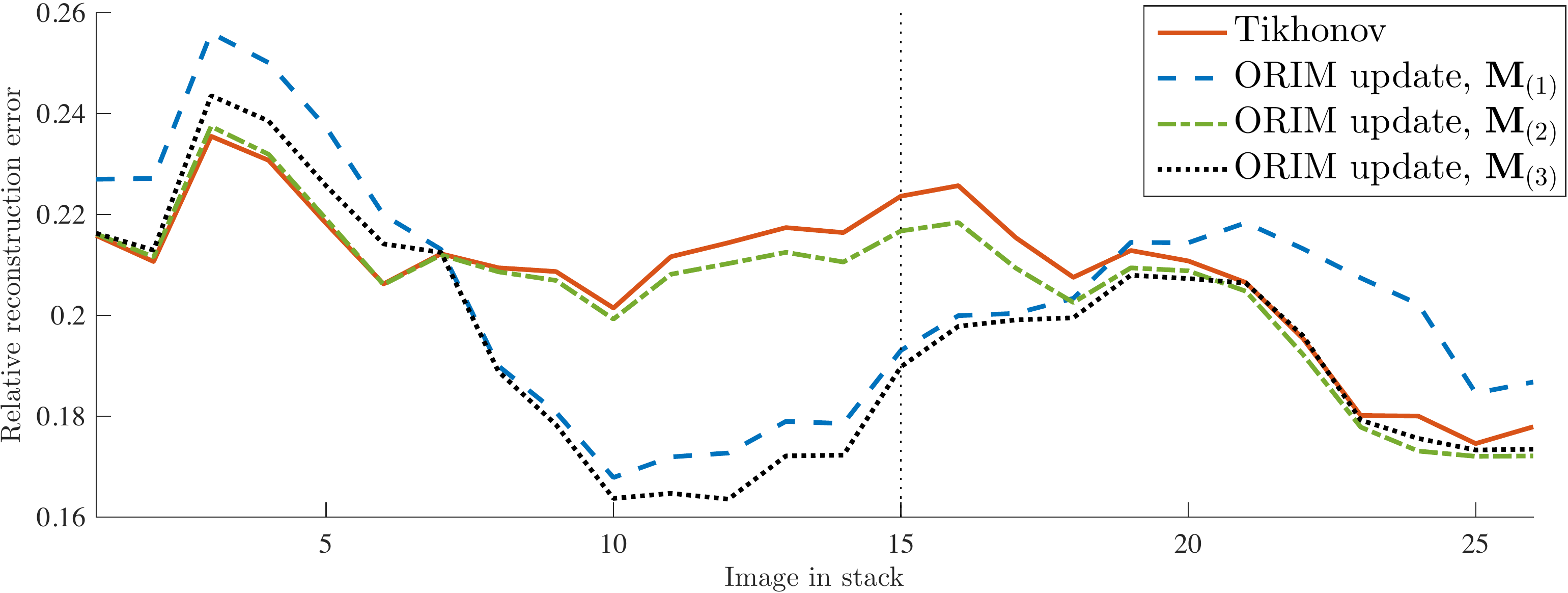}
	\end{center}
	\caption{Reconstructions of different slices from the MRI image stack using the initial Tikhonov reconstruction matrix, as well as the ORIM-updated reconstruction matrices.}
		\label{fig:mrislices}
\end{figure}
% subsection updating_inverse_matrices (end)

 \subsection{Experiment 3: ORIM updates for perturbed problems} \label{sub:experiment3} 
Last, we consider an example where ORIM updates to existing regularized inverse matrices can be used to efficiently solve perturbed problems.  That is, consider a linear inverse problem such as~\eqref{eqn:linearsystem} where a good regularized inverse matrix denoted by $\bfP$ can be obtained.  Now, suppose $\bfA$ is modified slightly (e.g., due to equipment setup or a change in model parameters), and a perturbed linear inverse problem
\begin{equation}
	\label{eqn:perturbed}
	\widetilde \bfb = \widetilde \bfA \bfxi + \widetilde \bfdelta
\end{equation}
must be solved. We will show that as long as the perturbation is not too large, a good solution to the perturbed problem can be obtained using low-rank ORIM updates to $\bfP$.  This is similar to the scenario described in Experiment 1, but here we use an example from 2D tomographic imaging, where the goal is to estimate an image or object $f(x,y)$, given measured projection data.  The Radon transform can be used to model the forward process, where the Radon transform of $f(x,y)$ is given by
\begin{equation}
	\label{eq:Radon}
	b(\xi,\phi) = \int f(x,y) \delta(x \cos\phi + y \sin\phi - \xi )\,\d x \,\d y
\end{equation}
where $\delta$ is the Dirac delta function.  Figure \ref{fig:Radon} illustrates the basic tomographic process.
\begin{figure}[bthp!]
	\begin{center}
		\includegraphics[width=.8\textwidth]{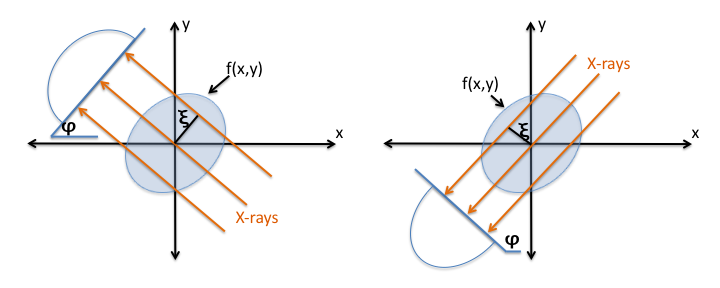}
	\end{center}
	\caption{Experiment 3: Illustration of 2D tomography problem setup, where $f(x,y)$ is the desired object and projection data is obtained by x-ray transmission at various angles around the object.}
		\label{fig:Radon}
\end{figure}

The goal of the inverse problem is to compute a (discretized) reconstruction of the image $f(x,y)$, given projection data that is collected at various angles around the object.  The projection data, when stored as an image, gives the sinogram.  In Figure~\ref{fig:tomo} (a), we provide the true image which is a $128 \times 128$ image of the Shepp-Logan phantom, and two sinograms are provided in Figure~\ref{fig:tomo} (b) and (c), where the rows of the image contain projection data at various angles. In particular, for this example, we take $60$ projection images at $3$ degree intervals from $0$ to $177$ degrees (i.e., the sinogram contains $60$ rows).  In order to deal with boundary artifacts, we pad the original image with zeros. 
\begin{figure}[bthp!]
	\begin{center}
		\begin{tabular}{ccc}
\includegraphics[width=.2\textwidth]{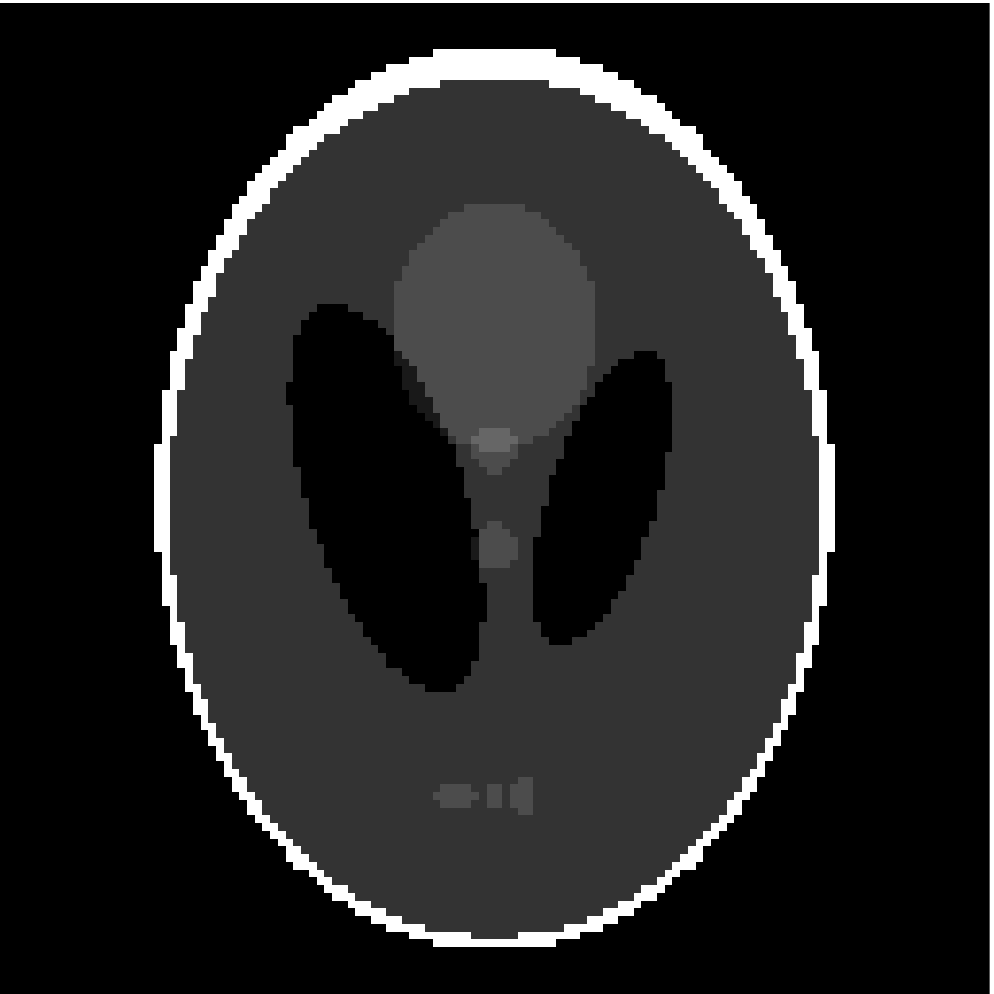} &
 \includegraphics[width=.35\textwidth]{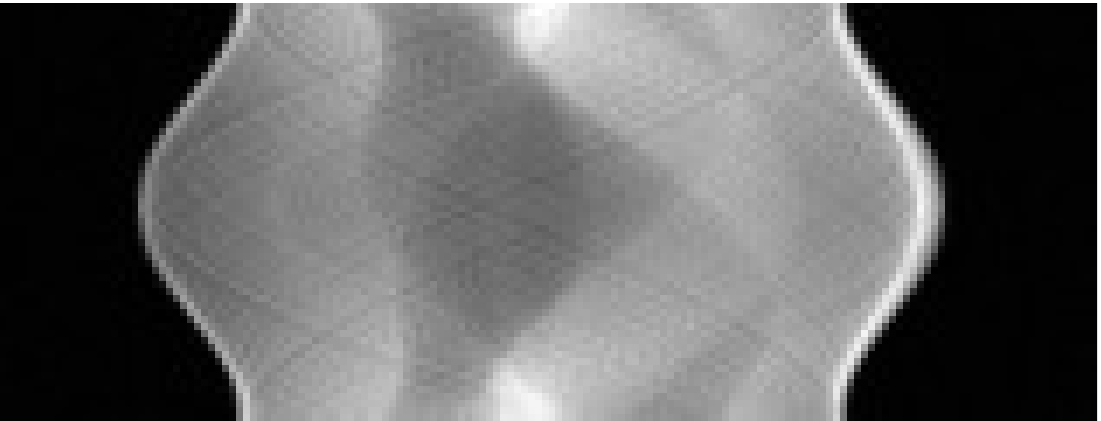}  &	
 \includegraphics[width=.35\textwidth]{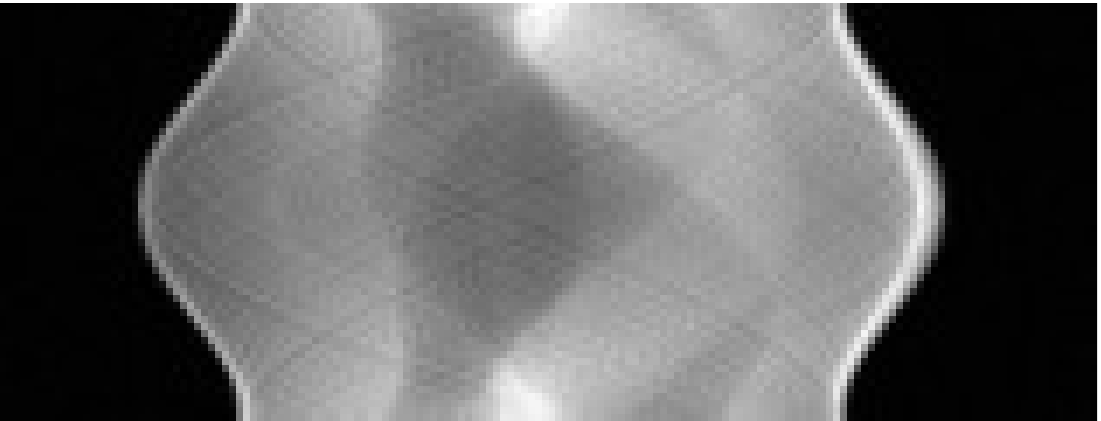}\\
  (a) True image & (b) Sinogram 1 & (c) Sinogram 2
		\end{tabular}
	\end{center}
	\caption{Tomography Problem. The true image is shown in (a), the observed sinogram for the initial problem is given in (b) and the sinogram corresponding to the perturbed problem is given in (c).}
		\label{fig:tomo}
\end{figure}

The discrete tomographic reconstruction problem can be modeled as~\eqref{eqn:linearsystem} where $\bfxi$ represents the (vectorized) desired image, $\bfA$ models the tomographic process, and $\bfb$ is the (vectorized) observed sinogram. For this example, we construct $$\bfA = \begin{bmatrix}\bfR \bfS_{(1)} \\ \vdots\\\bfR \bfS_{(60)} 
\end{bmatrix},$$ where $\bfS_{(j)}$ is a sparse matrix that represents rotation of the image for the $j$-th angle, whose entries were computed using bilinear interpolation as described in \cite{ChHaNa06,Chung2010b}, and $\bfR$ is a Kronecker product that approximates the integration operation.  It is worth mentioning that in typical tomography problems, $\bfA$ is never created, but rather accessed via projection and backprojection operations \cite{feeman2015mathematics}.  Our methods also work for scenarios where $\bfA$ represents a function call or object, but our current approach allows us to to build the sparse matrix directly.  White noise is added to the problem at relative noise level $0.005$.

Since $\bfA$ has no obvious structure to exploit, we use iterative reconstruction methods to get an initial reconstruction matrix.  This mimics a growing trend in tomography where reconstruction methods have shifted from filtered back projection approaches to iterative reconstruction methods \cite{hsieh2009computed,beister2012iterative}.  Furthermore, these iterative approaches are ideal for problems such as limited angle tomogography or tomosynthesis, where the goal is to obtain high quality images while reducing the amount of radiation to the patient \cite{dobbins2003digital,Chung2010}.
In this paper, we define a regularized inverse matrix $\bfP$ in terms of a partial Golub-Kahan bidiagonalization.  That is, given a matrix $\bfA$ and vector $\bfb,$ the Golub-Kahan process iteratively transforms matrix $[\bfb\,\,\, \bfA]$ to upper-bidiagonal form $[\beta_1 \bfe_1\,\,\, \bfB^{(k)}]$, with initializations $\beta_1 = \norm[2]{\bfb}$, $\bfw_1 = \bfb / \beta_1$ and $\alpha_1 \bfq_1 = \bfA\t \bfw_1$.  After $k$ steps of the Golub-Kahan bidiagonalization process, we have matrices $\bfQ^{(k)} = \begin{bmatrix}
	\bfq_1 & \ldots& \bfq_k
\end{bmatrix} \in \bbR^{n\times k}$,  $\bfW^{(k)} = \begin{bmatrix}
	\bfw_1 & \ldots& \bfw_k
\end{bmatrix} \in \bbR^{m\times k}$, and bidiagonal matrix
\begin{equation*}
 	\bfB^{(k)} = \begin{bmatrix} \alpha_1 & & &  \\
 	\beta_2 & \alpha_2 & &  \\
	 & \ddots & \ddots &  \\
	& & \beta_k & \alpha_k \\
	&	& & \beta_{k+1}\\
 	\end{bmatrix} \in \bbR^{(k+1)\times k},
 \end{equation*}
such that
\begin{equation}
	\bfA \bfQ^{(k)} = \bfW^{(k+1)} \bfB^{(k)}.
\end{equation}
It is worth noting that in exact arithmetic, the $k$-th LSQR \cite{PaSa82a,PaSa82b} iterate is given by $\bfx_{\rm LSQR} = \bfQ^{(k)} (\bfB^{(k)})^{\dagger} (\bfW^{(k+1)})\t \bfb$.  Thus, we define $\bfP = \bfQ^{(k)} (\bfB^{(k)})^{\dagger} (\bfW^{(k+1)})\t $ to be a regularized inverse matrix for the original problem, where $k=46$ corresponds to minimal reconstruction error $\norm[2]{\bfx_{\rm LSQR} - \bfx_{\rm true}}/ \norm[2]{\bfx_{\rm true}} = 0.2641$ for the original problem.  See Figure~\ref{fig:relerr} for the relative error plot for the original problem.  
\begin{figure}[bthp!]
	\begin{center}
				\includegraphics[width=\textwidth]{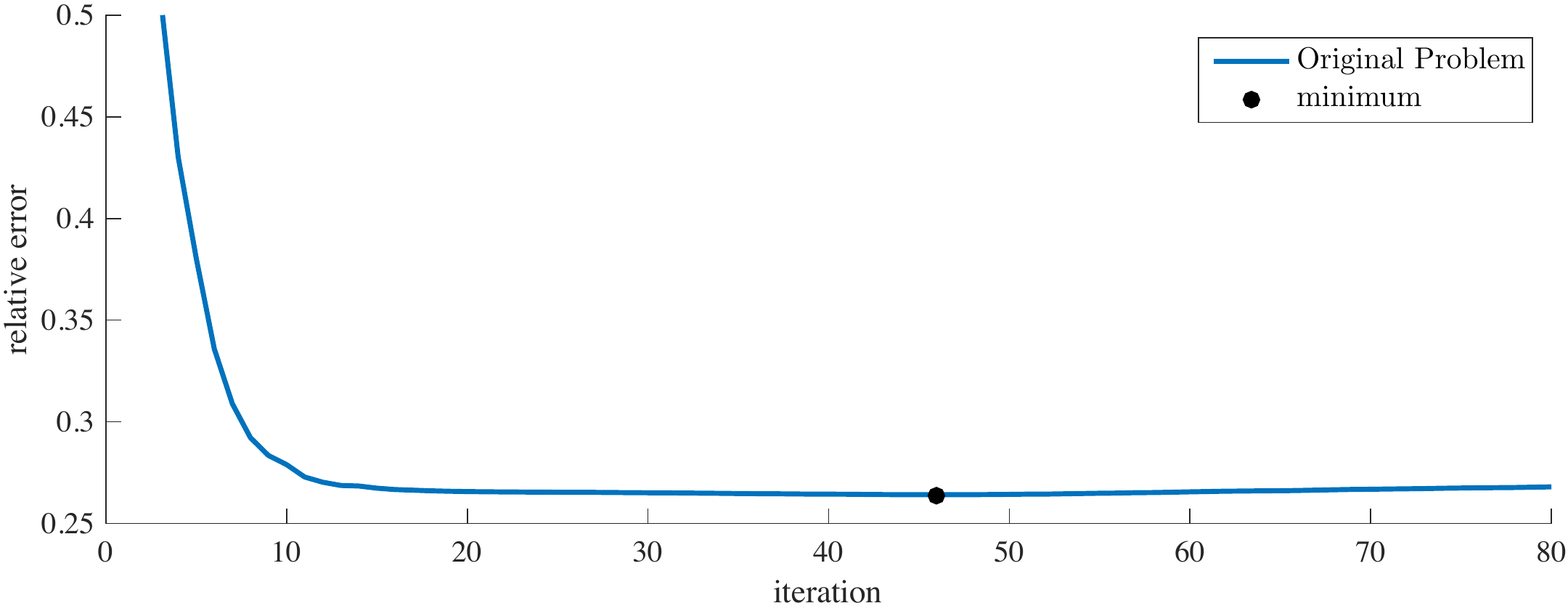}
	\end{center}
	\caption{Relative reconstruction errors for LSQR on the original tomography problem, where the bullet $\bullet$ corresponds to minimal reconstruction error.}
		\label{fig:relerr}
\end{figure}

The goal of this illustration is to show that a low-rank ORIM update to $\bfP$ can be used to solve a perturbed problem.  Thus, we created a perturbed problem~\eqref{eqn:perturbed}, where $\widetilde \bfb$ and $\widetilde\bfA$ were created with slightly shifted projection angles.   Again, we take $60$ projection images at $3$ degree intervals, but this time the angles ranged from $1$ to $178$ degrees.  The corresponding sinogram is given in Figure~\ref{fig:tomo}(c).  A first approach would be to use $\bfP$ to reconstruct the perturbed data: $\bfP \widetilde \bfb$.  This reconstruction is provided in the top left corner of Figure~\ref{fig:tomoresults}, and it is evident that this is not a very good reconstruction.  After a rank-4 update to $\bfP$, where $\bfmu_\bfxi = \bfzero_{n \times 1}$, $\bfM_\bfxi = \bfI_n$ and $\eta = 0.08$, we get a significantly better reconstruction (middle column of Figure~\ref{fig:tomoresults}).  For comparison purposes, we provide in the last column the best LSQR reconstruction for the perturbed problem (i.e., corresponding to minimal reconstruction error).
Relative reconstruction errors are provided, and corresponding absolute error images are presented on the same scale and with inverted colormap.
 
 \begin{figure}[bthp!]
 	\begin{center} 
		Initial, ${\rm rel} = 1.438$ \hspace{11ex} ORIM, ${\rm rel} = 0.287$ \hspace{11ex} LSQR, ${\rm rel} = 0.267$\\
 		\includegraphics[width=1\textwidth]{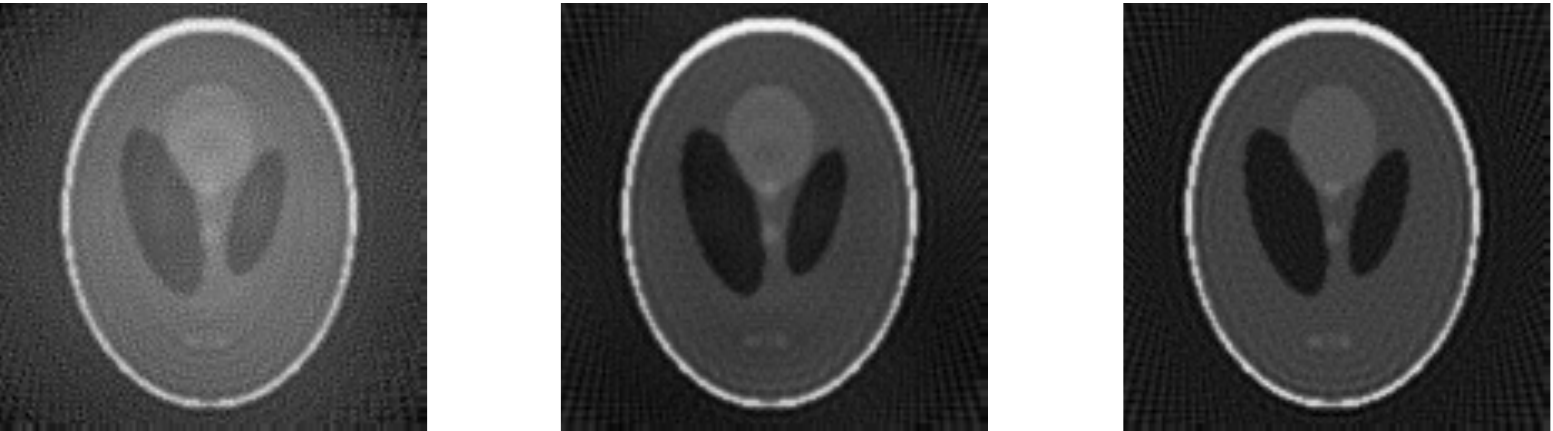} \\
 		\includegraphics[width=1\textwidth]{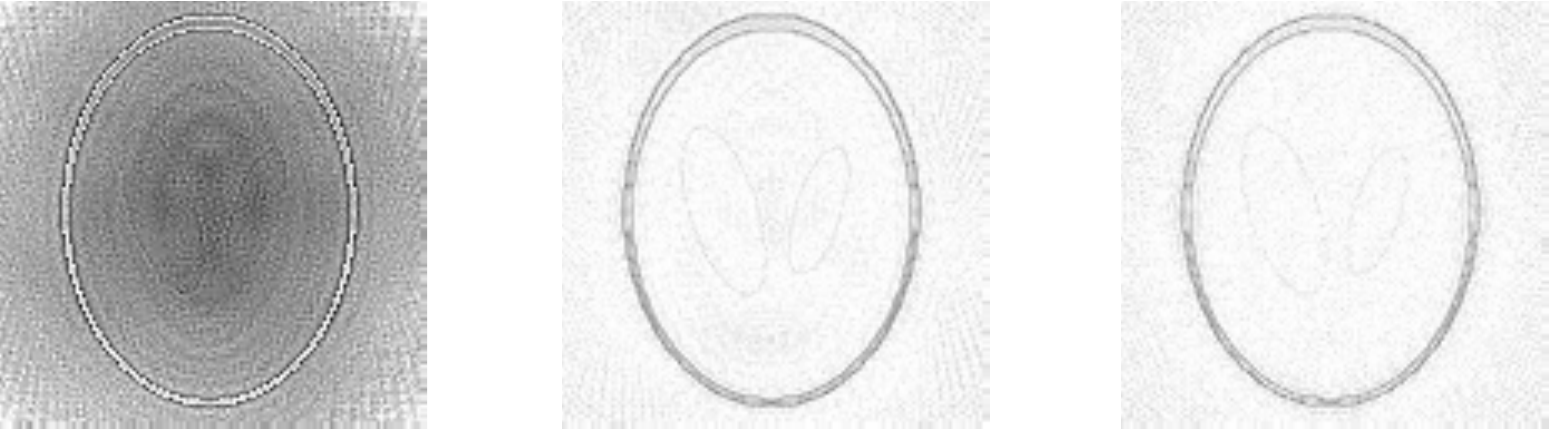}
 	\end{center}
 	\caption{Tomographic reconstructions for the perturbed problem, with corresponding error images.  The reconstruction in the first column was obtained as $\bfP \widetilde \bfb$, the reconstruction in the second column was obtained using a rank-4 ORIM update to P and was computed as $(\bfP+ \widehat \bfZ) \widetilde \bfb$.  The reconstruction in the last column corresponds to the LSQR reconstruction for the perturbed problem corresponding to minimal reconstruction error.}
 		\label{fig:tomoresults}
 \end{figure}

 \section{Conclusions} % (fold)
 \label{sec:conclusions}

 In this paper, we provide an explicit solution for a generalized rank-constrained matrix inverse approximation problem.   We define the solution to be an optimal regularized inverse matrix (ORIM), where we include regularization terms, rank constraints, and a more general weighting matrix. Two main distinctions from previous results are that we can include updates to an existing matrix inverse approximation, and in the Bayes risk minimization framework, we can incorporate additional information regarding the probability distribution of $\bfxi.$  For large scale problems, obtaining an ORIM according to Theorem~\ref{thm:mainresult} can be computationally prohibitive, so we described an efficient rank-update approach that decomposes the optimization problem into smaller rank subproblems and uses gradient-based methods that can exploit linearity.  Using examples from image processing, we showed that ORIM updates can be used to compute more accurate solutions to inverse problems and can be used to efficiently solve perturbed systems, which opens the door to new applications and investigations.  In particular, our current research is on incorporating ORIM updates within nonlinear optimization schemes such as variable projection methods, as well as on investigating its use for updating preconditioners for slightly changing systems.

\bibliographystyle{siamplain}
\bibliography{orim}
\end{document}